\documentclass[12pt,reqno]{amsart}

\usepackage{multirow}
\usepackage{hhline}

\usepackage{mathtools}
\usepackage{times}
\usepackage[T1]{fontenc}
\usepackage{mathrsfs}
\usepackage{latexsym}
\usepackage[dvips]{graphics}
\usepackage[titletoc, title]{appendix}
\usepackage{amsmath,amsfonts,amsthm,amssymb,amscd}
\usepackage{color}
\usepackage{hyperref}
\usepackage{amsmath}

\usepackage{color}
\usepackage{breakurl}

\usepackage{comment}
\newcommand{\bburl}[1]{\textcolor{blue}{\url{#1}}}

\newtheorem{thm}{Theorem}[section]

\newtheorem{cor}[thm]{Corollary}

\newtheorem{lem}[thm]{Lemma}
\newtheorem{prop}[thm]{Proposition}

\newtheorem{defi}[thm]{Definition}
\newtheorem{rek}[thm]{Remark}
\newtheorem{prob}[thm]{Problem}

\DeclareMathOperator{\supp}{supp}
\DeclareMathOperator{\spann}{span}
\DeclareMathOperator{\sgn}{sgn}
\numberwithin{equation}{section}

\begin{document}

\title[Property (A) Constants, Variations, and Lebesgue Inequalities for WTGA]{Variations of Property (A) Constants and Lebesgue-type Inequalities for the Weak Thresholding Greedy Algorithms}

\author{H\`ung Vi\d{\^e}t Chu}

\email{\textcolor{blue}{\href{mailto:hungchu2@illinois.edu}{hungchu2@illinois.edu}}}
\address{Department of Mathematics, University of Illinois at Urbana-Champaign, Urbana, IL 61820, USA}

\begin{abstract} Albiac and Wojtaszczyk introduced property (A) to characterize $1$-greedy bases. Later, Dilworth et al. generalized the concept to $C$-property (A), where the case $C = 1$ gives property (A). They (among other results) characterized greedy bases by unconditionality and $C$-property (A). In this paper, we extend the definition of the so-called A-property constant to (A,$\tau$)-property constants and use the extension to obtain new estimates for various Lebesgue parameters. Furthermore, we study the relation among (A,$\tau$)-property constants and other well-known constants when $\tau$ varies.
\end{abstract}



\thanks{The author is thankful to Timur Oikhberg for helpful feedback on earlier drafts of this paper.}

\maketitle

\tableofcontents

\section{Introduction}
\subsection{Lebesgue-type parameters}
Let $(\mathbb{X},\|\cdot\|)$ be an infinite-dimensional Banach space over the field $\mathbb{F} = \mathbb{R}$ or $\mathbb{C}$ with dual $(\mathbb{X}^*, \|\cdot\|_{*})$. We say that $\mathcal{B} = (e_n)_{n=1}^\infty$ is a \textit{semi-normalized Markushevich} basis (or M-basis or simply a basis, for short) if the following hold
\begin{enumerate}
    \item There exists a unique collection of biorthogonal functionals $(e^*_n)_{n=1}^\infty$ such that $e^*_i(e_j) = \delta_{i,j}$. 
    \item $0<\inf_{n}\{\|e_n\|, \|e^*_n\|_{*}\} \le \sup_{n} \{\|e_n\|, \|e^*_n\|_{*}\} < \infty$. 
    \item $\mathbb{X} = \overline{\spann\{e_n: n\in \mathbb{N}\}}$.
    \item $\mathbb{X}^* = \overline{\spann\{e^*_n: n\in \mathbb{N}\}}^{w^*}$.
\end{enumerate}
With respect to an M-basis $\mathcal{B}$, every $x\in\mathbb{X}$ can be represented by the formal sum (possibly divergent) $\sum_{n=1}^\infty e_n^*(x)e_n$ with $\lim_{n\rightarrow\infty}e_n^*(x)=0$. If, in addition, 
\begin{enumerate}
    \item [(5)] the partial sum operator $(S_m)_{m=1}^\infty$ defined as $S_m(x)=\sum_{n=1}^m e_n^*(x)e_n$ is uniformly bounded, i.e, $\sup_{m}\|S_m\|\le C$ for some $C > 0$, 
\end{enumerate}
we say that $\mathcal{B}$ is a \textit{Schauder} basis. In this case, we let $\mathbf K_b:= \sup_{m}\|S_m\|$, which is called the \textit{basis constant}.
We set $\supp(x):= \{n: e_n^*(x)\neq 0\}$, $\mathbb{X}_c = \{x\in \mathbb{X}: |\supp(x)|<\infty\}$, and $\mathbb N_0 = \mathbb{N}\cup \{0\}$.

\subsubsection{Thresholding Greedy Algorithm (TGA)} We wish to approximate each vector $x\in \mathbb{X}$ with a finite linear combination from the basis $\mathcal{B}$. Among all, the greedy algorithm is probably the most natural, which we now describe. Consider $x \sim \sum_{n=1}^\infty e_n^*(x)e_n$. A finite set $\Lambda\subset\mathbb{N}$ is called a \textit{greedy set} of $x$ if $\min_{n\in\Lambda} |e_n^*(x)|\ge \max_{n\notin \Lambda}|e_n^*(x)|$. For $m\in \mathbb{N}_0$, let 
$$\mathcal{G}(x) \ :=\ \{\Lambda\,:\, \Lambda\mbox{ is a greedy set of }x\}\mbox{ and }\mathcal{G}(x, m)\ :=\ \{\Lambda\in \mathcal{G}(x)\,:\, |\Lambda| = m\}.$$
A \textit{greedy operator} of order $m$ is defined as 
$$G_m(x)\ : =\ \sum_{n\in \Lambda_x}e_n^*(x)e_n, \mbox{ for some }\Lambda_x\in\mathcal{G}(x,m).$$
Let $\mathcal{G}_m$ denote the set of all greedy operators of order $m$, and $\mathcal{G} = \cup_{m\ge 1}\mathcal{G}_m$. We have $G_0(x)=0$ and $\mathcal{G}(x, 0) = \{\emptyset\}$. We capture the error term from the greedy algorithm by 
$$\gamma_m(x)\ :=\ \sup_{G_m\in\mathcal{G}_m}\|x-G_m(x)\|,$$
and quantify the efficiency of the greedy algorithm by comparing $\gamma_m(x)$ to
the smallest possible error of an arbitrary $m$-term approximation
$$\sigma_m(x) \ :=\ \inf\left\{\left\|x-\sum_{n\in A}a_ne_n\right\|\,:\, A\subset \mathbb{N}, |A|\le m, (a_n)\subset \mathbb{F}\right\}.$$
We can also compare $\gamma_m(x)$ to the smallest projection error
$$\widetilde{\sigma}_m(x) \ :=\ \inf\left\{\left\|x-P_A(x)\right\|\,:\, A\subset \mathbb{N}, |A| = m\right\},$$
where $P_A(x) = \sum_{n\in A}e_n^*(x)e_n$. We have $\gamma_0(x) = \sigma_0(x) = \widetilde{\sigma}_0(x) = \|x\|$.
In particular, for each $m\in \mathbb{N}_0$, let $\mathbf L_m$ (and $\mathbf{\widetilde{L}}_m$, respectively) be the smallest constant such that for all $x\in \mathbb{X}$, $$\gamma_m(x)\ \le\ \mathbf L_m\sigma_m(x)\mbox{ (and } \gamma_m(x)\ \le\ \mathbf{\widetilde{L}}_m\widetilde\sigma_m(x),\mbox{ respectively.)}$$
There have been many estimates and ongoing improvements for these Lebesgue parameters $\mathbf L_m$ and $\mathbf{\widetilde{L}}_m$ under various settings. For example, see \cite{AA, AW, BBG, DKOSZ, DBT, GHO, TYY}.

\subsubsection{Weak Thresholding Greedy Algorithm (WTGA)}
Throughout this paper, let $\tau$ be a number in $(0,1]$.  Temlyakov \cite{T} considered the WTGA, which allows more flexibility in forming greedy sums. Given $x\in \mathbb{X}$, a finite set $\Lambda\subset\mathbb{N}$ is \textit{$\tau$-weak greedy} with respect to $x$ if
$\min_{n\in \Lambda}|e^*_n(x)|\ \ge \ \tau\max_{n\notin \Lambda}|e^*_n(x)|$.
For $m\in \mathbb{N}_0$, let 
\begin{align*}
    \mathcal{G}(x,\tau) &\ :=\ \{\Lambda\,:\, \Lambda\mbox{ is a }\tau\mbox{-weak greedy set of }x\},\mbox{ and }\\
    \mathcal{G}(x, m,\tau)&\ :=\ \{\Lambda\in \mathcal{G}(x,\tau)\,:\, |\Lambda| = m\}.
\end{align*}
A $\tau$-\textit{greedy operator} of order $m$ is defined as 
$$G_m^\tau(x)\ : =\ \sum_{n\in \Lambda_x}e_n^*(x)e_n, \mbox{ for some }\Lambda_x\in\mathcal{G}(x,m,\tau).$$
Let $\mathcal{G}^\tau_m$ denote the set of all $\tau$-greedy operators of order $m$, and $\mathcal{G}^\tau = \cup_{m\ge 1}\mathcal{G}^\tau_m$. We capture the error term from the weak greedy algorithm by 
$$\gamma_{m,\tau}(x)\ :=\ \sup_{G^\tau_m\in\mathcal{G}^\tau_m}\|x-G_m^\tau(x)\|.$$
Similar to $\mathbf L_m$ and $\mathbf{\widetilde{L}}_m$, for $m \in \mathbb{N}_0$, we define $\mathbf L_{m,\tau}$ to be the smallest constant such that 
$$\gamma_{m,\tau}(x)\ \le\ \mathbf L_{m,\tau}\sigma_m(x), \forall x\in\mathbb{X},$$
and $\widetilde{\mathbf L}_{m,\tau}$ to be the smallest constant such that 
$$\gamma_{m,\tau}(x) \ \le\ \widetilde{\mathbf L}_{m,\tau}\widetilde{\sigma}_m(x), \forall x\in\mathbb{X}.$$
We call $\mathbf L_{m,\tau}$ the \textit{weak greedy Lebesgue parameter} and call $\widetilde{\mathbf L}_{m,\tau}$ the \textit{weak almost greedy Lebesgue parameter}.

\subsubsection{Chebyshev weak thresholding greedy algorithms (CWTGA)}

We also study the Lebesgue parameter for the CWTGA, which was introduced by Dilworth et al. \cite{DKK} to improve the rate of convergence of the TGA. A \textit{Chebyshev $\tau$-greedy operator} of order $m$ ($m\in \mathbb{N}_0$), denoted by $CG_{m}^\tau:\mathbb{X}\rightarrow\mathbb{X}$, satisfies the following:  for every $x\in\mathbb{X}$, there exists $\Lambda\in \mathcal{G}(x, m, \tau)$ such that
\begin{enumerate}
    \item $\supp(CG_{m}^\tau(x))\subset\Lambda$ and 
    \item we have $$\|x-CG_{m}^\tau(x)\|\ =\ \min \left\{\left\|x-\sum_{n\in\Lambda}a_ne_n\right\|\, :\, (a_n)\subset\mathbb{F}\right\}.$$
\end{enumerate}
For $m\in \mathbb{N}_0$, let $\mathbf L^{ch}_{m, \tau}$ denote the smallest constant such that 
$$\|x-CG^\tau_{m}(x)\|\ \le\ \mathbf L^{ch}_{m, \tau}\sigma_m(x), \forall x\in\mathbb{X}, \forall CG^\tau_{m}\in \mathcal{CG}_m^\tau,$$
where $\mathcal{CG}_m^\tau$ is the set of all Chebyshev $\tau$-greedy operators of order $m$ and $\mathcal{CG}^\tau = \cup_{m\ge 1}\mathcal{CG}_m^\tau$. We call $\mathbf L^{ch}_{m,\tau}$ the \textit{weak greedy Chebyshev Lebesgue parameter}. 

Tight estimates for both parameters $\mathbf{L}_{m,\tau}$ and $\mathbf{L}_{m,\tau}^{ch}$ were established in \cite{BBGHO, DKO}.

\subsubsection{Partially greedy parameter} The \textit{residual Lebesgue parameter} was introduced in \cite{DKO} to compare the performance of greedy algorithms $(G^\tau_m(x))_{m=1}^\infty$ to that of the partial sums $(S_m(x))_{m=1}^\infty$. Specifically, 
for $m\in \mathbb{N}_0$, let $\mathbf L^{re}_{m,\tau}$ be the smallest constant such that 
$$\gamma_{m,\tau}(x)\ \le\ \mathbf L^{re}_{m,\tau}\|x-S_m(x)\|,\forall x\in\mathbb{X}.$$
In \cite{BBL}, the authors established bounds for the so-called \textit{strong residual Lebesgue parameter}. For each $m\in \mathbb{N}_0$, they defined $\widehat{\mathbf{L}}^{re}_{m,\tau}$ as the smallest constant verifying
\begin{equation}\label{e11}\gamma_{m,\tau}(x) \ \le\ \widehat{\mathbf{L}}^{re}_{m,\tau}\widehat{\sigma}_m(x),\forall x\in\mathbb{X},\end{equation}
where $\widehat{\sigma}_m(x) = \inf_{0\le n\le m}\|x-S_n(x)\|$. Clearly, $\mathbf L^{re}_{m,\tau}\le \widehat{\mathbf{L}}^{re}_{m,\tau}$; if our basis is Schauder, then $(\mathbf K_b+1)\mathbf L^{re}_{m,\tau}\ge \widehat{\mathbf{L}}^{re}_{m,\tau}$, and so $\mathbf L^{re}_{m,\tau}\approx \widehat{\mathbf{L}}^{re}_{m,\tau}$. For estimates of $\mathbf L^{re}_{m,\tau}$ and $\widehat{\mathbf{L}}^{re}_{m,\tau}$, see \cite{BBL, DKO}.

\subsection{Notation and relevant constants}
For two functions $f(a_1, a_2, \ldots)$ and $g(a_1, a_2, \ldots)$, we write $f\lesssim g$ to indicate that there exists an absolute constant $C>0$ (independent of $a_1, a_2, \ldots$) such that $f\le Cg$. Similarly, $f\gtrsim g$ means that $Cf\ge g$ for some constant $C$. For two sets $A, B\subset\mathbb{N}$, we write $A<B$ for $\max A < \min B$ and write $A < m$ ($m\in \mathbb{N}$) for $\max A < m$. We shall use the following notation
$$1_A \ :=\ \sum_{n\in A}e_n \mbox{ and } 1_{\varepsilon A}\ :=\ \sum_{n\in A}\varepsilon_n e_n,$$
where $\varepsilon = (\varepsilon_n)_{n=1}^\infty$ and $|\varepsilon_n| = 1$. For $x\in \mathbb{X}$, $\|x\|_\infty := \sup_{n}|e_n^*(x)|$, and we write $A \sqcup B\sqcup x$ to indicate that $A, B$, and $\supp(x)$ are pairwise disjoint. We now present a list of relevant constants, most of which have appeared frequently in the literature, except that we may generalize them to accommodate the concept of $\tau$-weak greedy sets:
\begin{itemize}
    \item (A,$\tau$)-property constants: for $m\in \mathbb{N}_0$, 
    \begin{equation}\label{e7}\nu_{m,\tau} \ =\ \sup_{\substack{(\varepsilon), (\delta)\\\|x\|_\infty\le 1/\tau}}\left\{\frac{\|\tau x+1_{\delta B}\|}{\|x+1_{\varepsilon A}\|}\,:\, |A| = |B|\le m, A\sqcup B\sqcup x\right\}.\end{equation}
    In the above definition of $\nu_{m,\tau}$, we can replace the condition $``|A| = |B| \le m"$ by $``|B|\le |A| \le m"$ due to norm convexity.  
    \item Left (A, $\tau$)-property constants: for $m\in \mathbb{N}_0$,
    \begin{align}\label{e12}&\nu_{m,\tau, \ell}\ =\ \nonumber \\
    &\sup_{\substack{(\varepsilon), (\delta)\\ \|x\|_\infty\le 1/\tau}}\left\{\frac{\|\tau x+1_{\delta B}\|}{\|x+1_{\varepsilon A}\|}\,:\, |A| = |B|\le m,  B <  A, A\sqcup B\sqcup x\right\},\mbox{ and }\\
    &\nu'_{m,\tau, \ell}\ =\ \nonumber\\
    &\sup_{\substack{(\varepsilon), (\delta)\\ \|x\|_\infty\le 1/\tau}}\left\{\frac{\|\tau x+1_{\delta B}\|}{\|x+1_{\varepsilon A}\|}\,:\, |B| \le |A|\le m,  B <  \supp(x)\sqcup A, B\le m\right\}\label{e13}.
    \end{align}
    Note that while $\nu'_{m,\tau,\ell}$ puts more restrictions on the positions of sets $A$ and $B$ than $\nu_{m,\tau,\ell}$ does, $\nu'_{m,\tau,\ell}$ allows $A$ and $B$ to be of different cardinalities, but $\nu_{m,\tau, \ell}$ does not. 
    \item Unconditional constants: for $m\in \mathbb{N}_0$,
    $$k_m \ =\ \sup_{|A|\le m}\|P_A\|\mbox{ and } k_m^c\ =\ \sup_{|A|\le m}\|I-P_A\|.$$
    \item Quasi-greedy constants: for $m\in \mathbb{N}_0$,
    \begin{align*}g_{m,\tau}&\ =\ \sup\left\{\|G^\tau_n\|\,:\, G^\tau_n\in \cup_{k\le m}\mathcal{G}^\tau_k\right\},\mbox{ and }\\g^c_{m,\tau}&\ =\ \sup\left\{\|I-G^\tau_n\|\,:\, G^\tau_n\in \cup_{k\le m}\mathcal{G}^\tau_k\right\}.\end{align*}
    We have $g_{m,\tau}^c-1\le g_{m,\tau}\le g_{m,\tau}^c+1$.
    \item Super-democracy constants: for $m\in \mathbb{N}$,
        \begin{equation}\label{e50}\mu_m\ =\ \sup\left\{\frac{\|1_{\delta B}\|}{\|1_{\varepsilon A}\|}\,:\,|A| = |B| \le m, (\varepsilon),  (\delta)\right\}.\end{equation}
    \item Super-conservancy constants: for $m\in \mathbb{N}$,
            \begin{equation}\label{e22}\psi_m\ =\ \sup_{(\varepsilon), (\delta)}\left\{\frac{\|1_{\delta B}\|}{\|1_{\varepsilon A}\|}\,:\,|A| = |B| \le m, B < A\right\}.
            \end{equation}
\end{itemize}

\subsection{Main results}
Three main goals of this paper are to
\begin{itemize}
    \item use (left) (A,$\tau$)-property constants to generalize tight estimates for $\mathbf{L}_{m,\tau}$, $\widetilde{\mathbf{L}}_{m,\tau}$, $\mathbf{L}_{m,\tau}^{ch}$, $\mathbf L^{re}_{m,\tau}$, and $\widehat{\mathbf{L}}^{re}_{m,\tau}$.
    \item show how our estimates help generalize various existing results in the literature. 
    \item study the relation among (A,$\tau$)-property constants and other well-known constants when $\tau$ varies.
\end{itemize}

\begin{thm}\label{m1}
Let $\mathcal{B}$ be an M-basis in $\mathbb{X}$. Then for $m\in \mathbb{N}$,
\begin{equation}\label{e1} \max\left\{{k}^c_m, \widetilde{\mathbf L}_{m,\tau}, \frac{\nu_{m,\tau}}{\tau}\right\}\ \le\ \mathbf L_{m,\tau}\ \le\ \frac{k_{2m-1}^c\nu_{m,\tau}}{\tau}.\end{equation}
\end{thm}

\begin{thm}\label{m2}
Let $\mathcal{B}$ be an M-basis in $\mathbb{X}$. Then  for $m\in \mathbb{N}$,
\begin{equation}\label{e2} \max \left\{g^c_{m,\tau}, \frac{\nu_{m,\tau}}{\tau}\right\}\ \le\ \widetilde{\mathbf L}_{m,\tau}\ \le \ \frac{g^c_{m-1,\tau}\nu_{m,\tau}}{\tau}.\end{equation}
\end{thm}

\begin{prop}\label{p4}
For $m\in \mathbb{N}_0$, 
\begin{equation}\label{e18}\frac{\nu'_{m,\tau, \ell}}{\tau} \ \le\ \max_{0\le k\le m}\widehat{\mathbf L}^{re}_{k,\tau}.\end{equation}
\end{prop}

\begin{thm}\label{m3}
Let $\mathcal{B}$ be an M-basis in $\mathbb{X}$. Then
\begin{align}
    \mathbf L^{re}_{m,\tau}&\ \le\ \frac{g^c_{m-1,\tau}\nu_{m,\tau,\ell}}{\tau}, \forall m\in \mathbb{N}\label{e3},\\
    g^c_{m,\tau}\ \le\ \widehat{\mathbf L}^{re}_{m,\tau}&\ \le\ \frac{g^c_{m-1,\tau}\nu'_{m,\tau,\ell}}{\tau},\forall m\in \mathbb{N}.\label{e10}
\end{align}
As a result, 
\begin{equation}\label{e19}\widehat{\mathbf L}^{re}_{1,\tau} \ =\ \frac{\nu'_{1,\tau,\ell}}{\tau}.\end{equation}
If $\mathcal{B}$ is Schauder, 
then 
\begin{equation}\label{e15}
 \mathbf L^{re}_{m,\tau}\ \ge\ \frac{g^c_{m,\tau}}{\mathbf K_b+1},\forall m\in \mathbb{N}_0.
\end{equation}
\end{thm}

\begin{thm}\label{m4}
Let $\mathcal{B}$ be a Schauder basis with basis constant $\mathbf K_b$. Then for $m\ge \mathbb{N}_0$, 
\begin{equation}\label{e25}g^{c}_{m,\tau}\ \le\ \mathbf K_b \mathbf L^{ch}_{m,\tau}(1+\mathbf L^{ch}_{m,\tau}+\mathbf L^{ch}_{m,\tau}\mathbf K_b).\end{equation}
Consequently,
\begin{equation}\label{e26}\mathbf L^{ch}_{m,\tau}\ \ge\ \frac{1}{\mathbf K_b}\left(\frac{g^c_{m,\tau}}{3}\right)^{1/2}.\end{equation}
\end{thm}

\begin{rek}\normalfont
When $\tau = 1$, Theorems \ref{m1} and \ref{m2} give \cite[Inequalities (1.3) and (1.5)]{BBG}. Similarly, Proposition \ref{p4} implies \cite[Proposition 1.13]{BBL}. Finally, \eqref{e10} and \eqref{e19} give \cite[Theorem 1.14]{BBL}.
\end{rek}

We describe the outline of this paper: 
\begin{itemize}
    \item Section \ref{prel} presents important results that will be used in due course;
    \item Section \ref{bound} establishes bounds for Lebesgue parameters;
    \item Section \ref{sd} studies the relation among Lebesgue constants, (left) Property (A, $\tau$), and partial symmetry for largest coefficients;
    \item Section \ref{uniform} studies the relation among Property (A, $\tau$), uniform property (A), and quasi-greedy bases;
    \item Section \ref{cL} estimates Lebesgue constants when we go from the classical greedy setting to the weak greedy setting;
    \item Section \ref{future} contains several questions for further investigation. 
\end{itemize}

\section{Preliminary results}\label{prel}

We prove an analog of \cite[Proposition 2.1]{AW} for the unconditionality constant $k_m^c$.
\begin{prop}\label{p1}
Fix $m\in \mathbb{N}$. Let $J\subset\mathbb{N}$ with $|J| \le m$. Assume $(a_n)_{n\in J}, (b_n)_{n\in J}$ are scalars so that $|a_n|\le |b_n|$ for all $n\in J$, and $\sgn(a_n) = \sgn(b_n)$ whenever $a_nb_n\neq 0$. Then 
$$\left\|x+\sum_{n\in J} a_ne_n\right\|\ \le\ k^c_m \left\|x+\sum_{n\in J} b_ne_n\right\|,$$
for all $x\in \mathbb{X}$, where $\supp(x)\cap J = \emptyset$. If, in addition, there exists $j\in J$ such that 
$a_j = b_j$, then
$$\left\|x+\sum_{n\in J} a_ne_n\right\|\ \le\ k^c_{m-1} \left\|x+\sum_{n\in J} b_ne_n\right\|,$$
\end{prop}

\begin{proof}
For each $n\in J$, we have 
$$\frac{a_n}{b_n} \ =\ \int_{0}^{\frac{a_n}{b_n}}1dt;$$
hence,
$$x+\sum_{n\in J}a_n e_n\ =\ x+\sum_{n\in J}\left(\int_{0}^1b_n \chi_{(0, \frac{a_n}{b_n})}(t)dt\right) e_n\ =\ \int_{0}^1 \left(x+\sum_{n\in J}b_n\chi_{(0, \frac{a_n}{b_n})}(t)e_n\right)dt.$$
For each $t\in (0,1)$, we have
$$\left\|x+\sum_{n\in J}b_n\chi_{(0, \frac{a_n}{b_n})}(t)e_n\right\|\ \le\ k^c_{m}\left\|x+\sum_{n\in J}b_n e_n\right\|.$$
Therefore, 
$$\left\|x+\sum_{n\in J}a_n e_n\right\|\ \le\ \int_{0}^1 \left\|x+\sum_{n\in J}b_n\chi_{(0, \frac{a_n}{b_n})}(t)e_n\right\|dt\ \le\ k^c_{m}\left\|x+\sum_{n\in J}b_n e_n\right\|,$$
where in the last inequality, we can replace $k^c_m$ by $k^c_{m-1}$ if $a_j = b_j$ for some $j\in J$.
This completes our proof. 
\end{proof}

\begin{lem}\label{l1}
For $m\in \mathbb{N}_0$, let 
\begin{align}&\Omega_{m,\tau}\ : = \ \nonumber\\
&\sup_{\substack{(\varepsilon)\\\|x\|_\infty\le t/\tau}}\left\{\frac{\|x\|}{\|x-P_B(x) + t1_{\varepsilon A}\|}\,:\, |A|= |B| \le m, (\supp(x)\cup B)\cap A = \emptyset\right\}.\label{e5}\end{align}
We have $$\Omega_{m,\tau}\ =\ \frac{\nu_{m,\tau}}{\tau}, \forall m\in \mathbb{N}_0.$$
\end{lem}
\begin{proof}
Let $A, B, x, (\varepsilon), t$ be chosen as in \eqref{e5}. By norm convexity, we have
\begin{align*}
    \|x\|\ =\ \left\|x-P_B(x) + \sum_{n\in B}e_n^*(x)e_n\right\| &\ \le\ \sup_{(\delta)}\left\|x-P_B(x) + \frac{t}{\tau} 1_{\delta B}\right\|\\
    &\ =\ \frac{1}{\tau}\sup_{(\delta)}\left\|\tau(x-P_B(x)) + t 1_{\delta B}\right\|\\
    &\ \le\ \frac{\nu_{m,\tau}}{\tau}\left\|x-P_B(x) + t 1_{\varepsilon A}\right\|.
\end{align*}
Hence, $\Omega_{m,\tau}\le \nu_{m,\tau}/\tau$. 

For the reverse inequality, let $A, B, x, (\varepsilon), (\delta)$ be chosen as in \eqref{e7}. Let $y = \tau x + 1_{\delta B}$. Then $\|y\|_\infty \le 1$. By \eqref{e5}, 
$$\|\tau x + 1_{\delta B}\|\ =\ \|y\| \ \le\ \Omega_{m,\tau}\|y-P_B(y) + \tau 1_{\varepsilon A}\|\ =\ \Omega_{m,\tau}\tau \|x + 1_{\varepsilon A}\|.$$
Hence, $\Omega_{m,\tau}\ge \nu_{m,\tau}/\tau$.
\end{proof}

Similarly, we have two following lemmas, whose proofs are moved to the Appendix.

\begin{lem}\label{l2}
For $m\in \mathbb{N}_0$, let 
\begin{align}&\Omega_{m,\tau,\ell}\ : = \ \nonumber\\
&\sup_{\substack{(\varepsilon)\\ \|x\|_\infty\le t/\tau}}\left\{\frac{\|x\|}{\|x-P_B(x) + t1_{\varepsilon A}\|}\,:\, |A|= |B| \le m,  B < A,\supp(x)\cap A = \emptyset\right\}.\label{e8}\end{align}
We have $$\Omega_{m,\tau,\ell}\ =\ \frac{\nu_{m,\tau,\ell}}{\tau},\forall m\in \mathbb{N}_0.$$
\end{lem}

\begin{lem}\label{l6}
For $m\in \mathbb{N}_0$, let 
\begin{align}&\Omega'_{m,\tau,\ell}\ : = \ \nonumber\\
&\sup_{\substack{(\varepsilon)\\ \|x\|_\infty\le t/\tau\\ B\le m}}\left\{\frac{\|x\|}{\|x-P_B(x) + t1_{\varepsilon A}\|}\,:\, |B|\le |A| \le m,  B < \supp(x-P_B(x))\sqcup A\right\}.\label{e16}\end{align}
We have $$\Omega'_{m,\tau,\ell}\ =\ \frac{\nu'_{m,\tau,\ell}}{\tau}, \forall m\in \mathbb{N}_0.$$
\end{lem}

The next proposition shares the same spirit with Proposition \ref{p1} for the constant $g^c_{m,\tau}$.
\begin{prop}\label{p2}
Let $m\in \mathbb{N}$ and $\alpha\ge 0$. Then for any $x\in \mathbb{X}$ with $\|x\|_\infty\le \alpha/\tau$ and any $(\varepsilon)$, we have 
$$\|x+\alpha 1_{\varepsilon A}\|\ \le\ g^c_{m,\tau}\left\|x + \sum_{n\in A}\varepsilon_n a_n e_n\right\|,$$
for all real scalars $(a_n)$ with $a_n \ge \alpha$ and any $A\subset\mathbb{N}$ with $|A|\le m$ and $A\cap \supp(x) = \emptyset$. If, in addition, there exists $j\in A$ such that $a_j = \alpha$, then   
$$\|x+\alpha 1_{\varepsilon A}\|\ \le\ g^c_{m-1,\tau}\left\|x + \sum_{n\in A}\varepsilon_n a_n e_n\right\|.$$
\end{prop}

\begin{proof}
Let $z = \sum_{n\in A}\varepsilon_n a_n e_n$, $y = x + z$, and $\Lambda_{\alpha, s} = \{n\in A: |e_n^*(z)|>\alpha/s\}$ for each $s\in (0, 1]$. Clearly, $\Lambda_{\alpha, s}$ is a $\tau$-weak greedy set of $y$ and $|\Lambda_{\alpha, s}|\le |A| \le m$. We have
\begin{align*}
    \left\|x+\alpha 1_{\varepsilon A}\right\| &\ =\ \left\|x + \int_{0}^1\left(\sum_{n\in A}\chi_{\left[0, \frac{\alpha}{|e_n^*(z)|}\right]}(s)e_n^*(z)e_n\right)ds\right\|\\
    &\ =\ \left\|\int_{0}^1 \left(x + (I-P_{\Lambda_{\alpha, s}})z\right) ds\right\|\\
    &\ \le\ \int_{0}^1 \left\|x + (I-P_{\Lambda_{\alpha, s}})z\right\| ds\\
    &\ =\ \int_{0}^1 \left\|(I-P_{\Lambda_{\alpha, s}})y\right\|ds\ \le\ g^c_{m,\tau}\|y\|.
\end{align*}
The proof is completed. 

The second assertion is obvious because if there exists $j\in A$ such that $a_j = \alpha$, then $|\Lambda_{\alpha, s}|\le m-1$.
\end{proof}

\section{Bounds for Lebesgue parameters}\label{bound}
\begin{proof}[Proof of Theorem \ref{m1}]
First, we establish the upper bound for $\mathbf L_{m,\tau}$. Let $x\in\mathbb{X}$, $m\in \mathbb{N}$ and $A\in \mathcal G(x, m, \tau)$. Let $z= \sum_{n\in B}b_ne_n$ with $B\subset\mathbb{N}$, $|B| = m$, and $(b_n)_{n\in B}\subset \mathbb{F}$. Set $\alpha:= \min_{n\in A}|e_n^*(x)|$. By definition, $\|x-P_A(x)\|_\infty\le \alpha/\tau$. By Lemma \ref{l1}, we have
$$\|x-P_A(x)\|\ \le\ \frac{\nu_{m,\tau}}{\tau}\left\|x-P_A(x) - P_{B\backslash A}(x) + \alpha\sum_{n\in A\backslash B}\sgn(e_n^*(x))e_n\right\|.$$

Case 1: if $|A\cup B| \le 2m-1$, then by the first assertion of Proposition \ref{p1}, we have
\begin{equation}\label{e14}\|x-P_A(x)\|\ \le\ \frac{\nu_{m,\tau} k^c_{2m-1}}{\tau}\left\|P^c_{A\cup B}(x) + \sum_{n\in A\backslash B}e_n^*(x)e_n + \sum_{n\in B}(e_n^*(x)-b_n)e_n\right\|.\end{equation}

Case 2: if $|A\cup B| = 2m$, then $A$ is disjoint from $B$. By the second assertion of Proposition \ref{p1}, we also have
\eqref{e14}

Taking the infinum over all $B$ and $(b_n)_{n\in B}$ to obtain $$\|x-P_A(x)\|\ \le\ \frac{\nu_{m,\tau} k^c_{2m}}{\tau}\sigma_m(x).$$ 
As $A$ is arbitrary, $\gamma_{m,\tau}(x)\le \frac{\nu_{m,\tau} k^c_{2m-1}}{\tau}\sigma_m(x)$ and so, $\mathbf L_{m, \tau}\le \nu_{m,\tau} k^c_{2m-1}/\tau$.

We now bound $\mathbf L_{m,\tau}$ from below. Observe that $\mathbf L_{m,\tau}\ge \mathbf L_{m}$ and $\mathbf L_{m}\ge k_m^c$ by \cite[Proposition 1.1]{BBG}; hence, $\mathbf L_{m,\tau}\ge k_m^c$. By definition, $\mathbf L_{m,\tau}\ge \widetilde{\mathbf L}_{m,\tau}$. As we shall show in the proof of Theorem \ref{m2}, $\widetilde{\mathbf L}_{m,\tau}\ge \nu_{m,\tau}/\tau$. Hence, 
$$L_{m,\tau}\ \ge\ \max\left\{k_m^c, \widetilde{\mathbf L}_{m,\tau}, \frac{\nu_{m,\tau}}{\tau}\right\}.$$
\end{proof}

\begin{proof}[Proof of Theorem \ref{m2}]
First, we establish the upper bound for $\widetilde{\mathbf L}_{m,\tau}$. Let $x\in\mathbb{X}$, $m\in \mathbb{N}$ and $A\in\mathcal G(x, m,\tau)$. Let $B\subset\mathbb{N}$ with $|B| = m$. Set $\alpha:= \min_{n\in A}|e_n^*(x)|$. By definition, $\|x-P_A(x)\|_\infty\le \alpha/\tau$. 
By Lemma \ref{l1}, we have
\begin{align*}\|x-P_A(x)\|&\ \le\ \frac{\mathbf \nu_{m,\tau}}{\tau}\left\|x-P_A(x) - P_{B\backslash A}(x) + \alpha\sum_{n\in A\backslash B}\sgn(e_n^*(x))e_n\right\|\\
&\ =\ \frac{\nu_{m,\tau}}{\tau}\left\|P^c_{A\cup B}(x) + \sum_{n\in A\backslash B}\frac{\alpha}{|e^*_n(x)|}e_n^*(x)e_n\right\|.
\end{align*}

Case 1: if $|A\backslash B|\le m-1$, then the first assertion of Proposition \ref{p2} gives
\begin{equation}\label{e6}\|x-P_A(x)\|\ \le\ \frac{\nu_{m,\tau}g^c_{m-1,\tau}}{\tau}\left\|P^c_{A\cup B}(x) + \sum_{n\in A\backslash B}e_n^*(x)e_n\right\|\ =\ \frac{\nu_{m,\tau}g^c_{m-1,\tau}}{\tau}\left\|x-P_B(x)\right\|.\end{equation}

Case 2: if $|A\backslash B| = m$, then $A\backslash B = A$, and the second assertion of Proposition \ref{p2} again gives \eqref{e6}. 

Taking the infinum over all $B$, we obtain $$\|x-P_A(x)\| \ \le \ \frac{\nu_{m,\tau}g^c_{m-1,\tau}}{\tau}\widetilde{\sigma}_m(x).$$
Since $A$ is arbitrary, we have proved $\widetilde{\mathbf L}_{m,\tau}\le \nu_{m,\tau}g^c_{m-1,\tau}/\tau$. 

Next, we bound $\widetilde{\mathbf L}_{m,\tau}$ from below. Choose $A, B, x, (\varepsilon), (\delta)$ as in \eqref{e7} with $x\in \mathbb{X}_c$. Set 
$$z\ :=\ 1_{\varepsilon A}+ x + \frac{1}{\tau}1_{\delta B} + 1_C,$$
where $ C > A\cup B\cup \supp(x)$ and $|C| = m - |A|$. Since $A\cup C\in \mathcal G(x, m, \tau)$,
\begin{align*}\left\|x+\frac{1}{\tau}1_{\delta B}\right\|\ =\ \left\|z-P_{A\cup C}(z)\right\|&\ \le\ \widetilde{\mathbf L}_{m,\tau}\tilde{\sigma}_{m}(z)\\
&\ \le\  \widetilde{\mathbf L}_{m,\tau}\left\|z-\frac{1}{\tau}1_{\delta B}-1_C\right\|\ =\ \widetilde{\mathbf L}_{m,\tau}\|1_{\varepsilon A}+x\|.\end{align*}
Therefore,  $$\frac{\left\|\tau x+1_{\delta B}\right\|}{\|x+1_{\varepsilon A}\|}\ \le\ \tau \widetilde{\mathbf L}_{m,\tau}.$$
We conclude that $\nu_{m,\tau}/\tau\le \widetilde{\mathbf L}_{m,\tau}$. Finally, we show that $\widetilde{\mathbf L}_{m,\tau}\ge g_m^c$. Let $x\in \mathbb{X}_c$ and $\Lambda\in \mathcal G(x, k, \tau)$ with $k\le m$. It suffices to prove $\|x- P_{\Lambda}(x)\|\le \widetilde{\mathbf L}_{m,\tau}\|x\|$. Let $\alpha = \min_{n\in \Lambda}|e_n^*(x)|$. Choose $C > \supp(x)$ with $|C| = m-k$. Set $y := x+1_C$.
Clearly, $\Lambda\cup C\in \mathcal G(y, m, \tau)$ and so,
\begin{align*}\|x-P_{\Lambda}(x)\|&\ =\ \|y - P_{\Lambda\cup C}(y)\|\ \le\ \widetilde{\mathbf L}_{m,\tau}\widetilde{\sigma}_m(y)\\
&\ \le\ \widetilde{\mathbf L}_{m,\tau}\|y - P_{C}(y)\|\ =\ \widetilde{\mathbf L}_{m,\tau}\|x\|,\end{align*}
as desired.
\end{proof}

\begin{proof}[Proof of Proposition \ref{p4}]
We follow the argument in the proof of \cite[Proposition 1.13]{BBL}. Let $A, B, x, (\varepsilon), (\delta)$ be chosen as in \eqref{e13}. Let $m_1 := \max B$. By the conditions put on sets $A$ and $B$, there exists a set $D$ (possibly empty) such that $\max D\le m_1, D\cap B = \emptyset$, and $m_1\le m_2:= |D\cup A|\le m$. Set 
$$y\ :=\  \frac{1}{\tau} 1_{\delta B}+ 1_D + x + 1_{\varepsilon A}.$$
Clearly, $D\cup A\in \mathcal G(y, m_2, \tau)$ and $S_{m_1}(y) = 1_D+\frac{1}{\tau}1_{\delta B}$. We, therefore, have
$$\left\|x+\frac{1}{\tau} 1_{\delta B}\right\| \ =\ \|y-P_{D\cup A}(y)\| \ \le\ \widehat{\mathbf L}^{re}_{m_2,\tau}\|y-P_{m_1}(y)\|\ \le\ \max_{0\le k\le m}\widehat{\mathbf L}_{k,\tau}\|x+1_{\varepsilon A}\|.$$
We conclude that $\nu'_{m,\tau,\ell}/\tau\le \max_{0\le k\le m}\widehat{\mathbf L}_{k,\tau}$.
\end{proof}

\begin{proof}[Proof of Theorem \ref{m3}]
We prove each inequality below.
\begin{enumerate}
\item Proof of \eqref{e3}: We first prove the right inequality. Let $x\in \mathbb{X}$, $m\in \mathbb{N}$, $A\in \mathcal G(x,m,\tau)$. Set $B = \{1, \ldots, m\}\backslash A$, $F = A\backslash \{1, \ldots, m\}$, and $\alpha:= \min_{n\in A}|e_n^*(x)|$. By definition, $\|x-P_A(x)\|_\infty\le \alpha/\tau$. Observe that $|B| = |F|$ and $B<F$. By Lemma \ref{l2} and Proposition \ref{p2}, we  have 
\begin{align*}
   &\|x-P_A(x)\|\\
    &\ \le\ \frac{\mathbf \nu_{m,\tau,\ell}}{\tau}\left\|x-P_A(x) - P_{B}(x) + \alpha\sum_{n\in F}\sgn(e_n^*(x))e_n\right\|\nonumber\\
    &\ =\ \frac{\nu_{m,\tau,\ell}}{\tau}\left\|P^c_{A\cup B}(x) + \sum_{n\in F}\frac{\alpha}{|e^*_n(x)|}e_n^*(x)e_n\right\|\nonumber\\
    &\ \le\ \frac{\nu_{m,\tau,\ell}g^c_{m-1,\tau}}{\tau}\left\|P^c_{A\cup B}(x) + \sum_{n\in F}e_n^*(x)e_n\right\|\ =\ \frac{\nu_{m,\tau,\ell}g^c_{m-1,\tau}}{\tau}\left\|x-S_m(x)\right\|,
\end{align*}
where we obtain the second inequality by case analysis as in the proof of Theorems \ref{m1} and \ref{m2}. Therefore, $\mathbf L^{re}_{m,\tau}\ \le\  g^c_{m-1,\tau}\nu_{m,\tau,\ell}/\tau$.

\item Proof of \eqref{e10}: The left inequality is immediate from \eqref{e11} by noticing that $\widehat{\sigma}_m(x)\le\|x\|$. We prove the right inequality. Let $x\in \mathbb{X}$, $m\in \mathbb{N}$, $A\in \mathcal G(x,m,\tau)$. Fix $k\le m$. Set $B = \{1, \ldots, k\}\backslash A$, $F = A\backslash \{1, \ldots, k\}$, and $\alpha:= \min_{n\in A}|e_n^*(x)|$. By definition, $\|x-P_A(x)\|_\infty\le \alpha/\tau$. Observe that $|B| \le |F|\le m$ and $B<\supp(x-P_A(x)-P_B(x))\sqcup F$. By Lemma \ref{l6} and Proposition \ref{p2}, we  have 
\begin{align*}
   &\|x-P_A(x)\|\\
    &\ \le\ \frac{\nu'_{m,\tau,\ell}}{\tau}\left\|x-P_A(x) - P_{B}(x) + \alpha\sum_{n\in F}\sgn(e_n^*(x))e_n\right\|\nonumber\\
    &\ =\ \frac{\nu'_{m,\tau,\ell}}{\tau}\left\|P^c_{A\cup B}(x) + \sum_{n\in F}\frac{\alpha}{|e^*_n(x)|}e_n^*(x)e_n\right\|\nonumber\\
    &\ \le\ \frac{\nu'_{m,\tau,\ell}g^c_{m-1,\tau}}{\tau}\left\|P^c_{A\cup B}(x) + \sum_{n\in F}e_n^*(x)e_n\right\|\ =\ \frac{\nu'_{m,\tau,\ell}g^c_{m-1,\tau}}{\tau}\left\|x-S_k(x)\right\|,
\end{align*}
where we obtain the second inequality by case analysis as in the proof of Theorems \ref{m1} and \ref{m2}. 
Therefore, $\widehat{\mathbf L}^{re}_{m,\tau}\ \le\  g^c_{m-1,\tau}\nu'_{m,\tau,\ell}/\tau$.
\item Proof of \eqref{e19}: This inequality follows directly from \eqref{e18} and \eqref{e10}. 
\item Proof of \eqref{e15}: This inequality is due to $\widehat{\mathbf L}^{re}_{m,\tau}\le (\mathbf K_b+1)\mathbf L^{re}_{m,\tau}$ and \eqref{e10}.
\end{enumerate}
\end{proof}

\begin{proof}[Proof of Theorem \ref{m4}]
Fix $x\in \mathbb{X}_c$, $m\in \mathbb{N}_0$, and $A\in \mathcal G(x, m, \tau)$. Choose $D\subset\mathbb{N}$ with $D > \supp(x)$, $|D| = m$ and set $y:= x-P_A(x) + \alpha 1_D$ and $z:= x+\alpha 1_D$, where $\alpha := \min_{n\in A}|e_n^*(x)|$. Then $D\in \mathcal G(y, m, \tau)$ and $A\in \mathcal G(z, m,\tau)$. Choose $CG_m^\tau \in \mathcal{CG}_m^\tau$ such that $\supp(CG_m^\tau(y))\subset D$ and $\supp(CG_m^\tau(z))\subset A$.
We have
$$\|y- CG_m^\tau(y)\|\ =\ \left\|x-P_A(x) + \sum_{n\in D} a_n e_n\right\|,$$
for some scalars $(a_n)\subset \mathbb{F}$. Hence, 
\begin{align}\label{e39}\|x-P_A(x)\|&\ \le\  \mathbf{K}_b \|y- CG_m^\tau(y)\|\ \le\ \mathbf{K}_b\mathbf L^{ch}_{m,\tau}\sigma_m(y)\nonumber\\
&\ \le\ \mathbf{K}_b\mathbf L^{ch}_{m,\tau}\|x+\alpha 1_D\|\ \le\ \mathbf{K}_b \mathbf L^{ch}_{m,\tau}(\|x\|+\|\alpha 1_D\|).
\end{align}
We now bound $\|\alpha 1_D\|$:
$$\|z-CG_m^\tau(z)\|\ =\ \left\|\sum_{n\in A}b_n e_n + \sum_{n\notin  A} e_n^*(x)e_n + \alpha 1_D\right\|,$$
for some scalars $(b_n)\subset \mathbb{F}$.
Therefore, 
\begin{align}\label{e40}
    \|\alpha 1_D\|\ \le\ (\mathbf K_b + 1)\left\|z-CG_m^\tau(z)\right\|\ \le\ (\mathbf K_b + 1)\mathbf L^{ch}_{m,\tau}\sigma_m(z)\ \le\ (\mathbf K_b+1)\mathbf L^{ch}_{m,\tau}\|x\|.
\end{align}
From \eqref{e39} and \eqref{e40}, we conclude that 
\begin{equation}\label{e42}\|x-P_A(x)\|\ \le\ \mathbf K_b \mathbf L^{ch}_{m,\tau}(1+\mathbf L^{ch}_{m,\tau}+\mathbf L^{ch}_{m,\tau}\mathbf K_b)\|x\|.\end{equation}
Hence, $g^{c}_{m,\tau}\le \mathbf K_b \mathbf L^{ch}_{m,\tau}(1+\mathbf L^{ch}_{m,\tau}+\mathbf L^{ch}_{m,\tau}\mathbf K_b)$, which is \eqref{e25}. To obtain $\eqref{e26}$, we use the trivial estimate, $$\mathbf K_b \mathbf L^{ch}_{m,\tau}(1+\mathbf L^{ch}_{m,\tau}+\mathbf L^{ch}_{m,\tau}\mathbf K_b)\ \le\ 3(\mathbf K_b \mathbf L^{ch}_{m,\tau})^2.$$
This completes our proof. 
\end{proof}

\section{Lebesgue constants, Property (A, $\tau$), and partial symmetry}\label{sd}
In this section, using our bounds for Lebesgue parameters, we generalize several existing theorems in the literature. We write ``$\sup_m$" to indicate ``$\sup_{m\ge 0}$".

\subsection{Estimates on the $\tau$-greedy constant (Generalization of \cite[Theorem 3.4]{AW} and \cite[Theorem 2]{DKOSZ})}

\begin{defi}\normalfont
A basis $\mathcal{B}$ is said to be $(C, \tau)$-greedy if $C:= C(\tau) \ge 1$ verifies
\begin{equation}\label{e27}\sup_{m} \mathbf L_{m,\tau} \ \le\ C \  <\ \infty .\end{equation}
The least constant satisfying \eqref{e27} is denoted by $\mathbf C_{g,\tau}$. In particular, 
$$\mathbf C_{g,\tau}\ =\ \sup_m \mathbf L_{m,\tau} \ <\ \infty.$$
\end{defi}

\begin{defi}\normalfont
A basis $\mathcal{B}$ is said to be $K$-suppression unconditional if $K\ge 1$ verifies
\begin{equation}\label{e28} \sup_{m}k_m \ \le\ K\ <\ \infty.\end{equation}
The least constant satisfying \eqref{e28} is denoted by $\mathbf K_s$; that is, 
$$\mathbf K_s\ =\ \sup_{m}k_m \ =\ \sup_{m}k_m^c \ <\ \infty.$$
\end{defi}

\begin{defi}\label{d31}\normalfont
A basis $\mathcal{B}$ is said to have $C$-property (A, $\tau$) if $C:= C(r)\ge 1$ verifies 
\begin{equation}\label{e29}\sup_{m} \nu_{m,\tau}\ \le\ C \ <\ \infty.\end{equation}
The least constant satisfying \eqref{e29} is denoted by $\mathbf C_{b,\tau}$; specifically, 
$$\mathbf C_{b,\tau}\ =\ \sup_{m}\nu_{m,\tau} \ <\ \infty.$$
\end{defi}

\begin{thm}\label{m6}
Let $\mathcal{B}$ be an M-basis. The following hold.
\begin{enumerate}
    \item If $\mathcal{B}$ is $(\mathbf C_{g,\tau}, \tau)$-greedy, then $\mathcal{B}$ is $\mathbf C_{g,\tau}$-suppression unconditional and has $\tau \mathbf C_{g,\tau}$-property (A, $\tau$). 
    \item Conversely, if $\mathcal{B}$ is $\mathbf K_s$-suppression unconditional and has $\mathbf C_{b,\tau}$-property (A, $\tau$), then $\mathcal B$ is $(\frac{\mathbf K_s\mathbf C_{b,\tau}}{\tau}, \tau)$-greedy.
\end{enumerate}
\end{thm}

\begin{rek}\normalfont
Observe that setting $\mathbf C_{g,\tau} = \tau = \mathbf C_{b,\tau} = \mathbf K_s = 1$ in Theorem \ref{m6} gives \cite[Theorem 3.4]{AW}. Instead, if we set $\tau = 1$, we have \cite[Theorem 2]{DKOSZ}. 
\end{rek}

\begin{proof}[Proof of Theorem \ref{m6}]
(1) By \eqref{e1}, if $\mathcal B$ is $(\mathbf C_{g,\tau}, \tau)$-greedy, 
\begin{align*}\mathbf C_{g,\tau} &\ =\ \sup_{m} \mathbf L_{m,\tau} \ \ge\ \sup_{m} \frac{\nu_{m,\tau}}{\tau}\ =\ \frac{\mathbf C_{b,\tau}}{\tau},\mbox{ and }\\
\mathbf C_{g,\tau}&\ =\ \sup_{m} \mathbf L_{m,\tau}\ \ge\ \sup_{m} k_m^c\ =\ \mathbf K_s. 
\end{align*}
Therefore, $\mathcal{B}$ is $\mathbf C_{g,\tau}$-suppression unconditional and has $\tau \mathbf C_{g,\tau}$-property (A, $\tau$).

(2) By \eqref{e1}, if $\mathcal B$ is $\mathbf K_s$-suppression unconditional and has $\mathbf C_{b,\tau}$-property (A, $\tau$),  
$$\sup_{\ge 1} \mathbf L_{m,\tau}\ \le\ \sup_{m\ge 1} \frac{k^c_{2m-1}\nu_{m,\tau}}{\tau}\ \le\ \frac{1}{\tau} \sup_{m\ge 1} k^c_{2m-1}\sup_{m\ge 1}\nu_{m,\tau}\ \le\ \frac{\mathbf K_s \mathbf C_{b,\tau}}{\tau}.$$
Hence, $\mathcal{B}$ is $(\frac{\mathbf K_s\mathbf C_{b,\tau}}{\tau}, \tau)$-greedy.
\end{proof}

\subsection{Estimates on the $\tau$-almost greedy constant (Generalization of \cite[Theorem 3.3]{AA})}
\begin{defi}\normalfont
A basis is said to be $(C,\tau)$-almost greedy if $C:= C(\tau)\ge 1$ verifies
\begin{equation}
    \label{e31} \sup_{m} \widetilde{\mathbf L}_{m,\tau}\ \le\ C\ <\ \infty.
\end{equation}
The least constant $C$ satisfying \eqref{e31} is denoted by $\mathbf C_{a, \tau}$, which is equal to 
$\sup_m \widetilde{\mathbf L}_{m,\tau}$. 
\end{defi}

\begin{defi}\normalfont
A basis is said to be $(C,\tau)$-quasi-greedy if $C := C(\tau) \ge 1$ verifies
\begin{equation}
    \label{e30} \sup_{m} g_{m,\tau}^c\ \le\ C\ <\ \infty. 
\end{equation}
The least constant $C$ satisfying \eqref{e30} is denoted by $\mathbf C_{\ell, \tau}$, which is equal to 
$\sup_{m} g_{m,\tau}^c$. Note that the term ``suppression quasi-greedy" is also used when $\tau = 1$ (see Definition \ref{d30}).
\end{defi}

\begin{thm}\label{m7}
Let $\mathcal{B}$ be an M-basis. The following hold.
\begin{enumerate}
    \item If $\mathcal{B}$ is $(\mathbf C_{a,\tau}, \tau)$-almost greedy, then $\mathcal{B}$ is $(\mathbf C_{a,\tau}, \tau)$-quasi-greedy and has $\tau \mathbf C_{\alpha, \tau}$-property (A, $\tau$).
    \item Conversely, if $\mathcal{B}$ is $(\mathbf C_{\ell, \tau}, \tau)$-quasi-greedy and has $\mathbf C_{b,\tau}$-property (A, $\tau$), then $\mathcal{B}$ is $(\frac{\mathbf C_{\ell, \tau}\mathbf C_{b,\tau}}{\tau},\tau)$-almost greedy. 
\end{enumerate}
\end{thm}

\begin{rek}\normalfont
Setting $\tau = 1$ in Theorem \ref{m7}, we obtain \cite[Theorem 3.3]{AA}.
\end{rek}

\begin{proof}[Proof of Theorem \ref{m7}]
(1) By \eqref{e2}, if $\mathcal B$ is $(\mathbf C_{a,\tau},\tau)$-almost greedy, 
\begin{align*}\mathbf C_{a,\tau}&\ =\ \sup_{m} \widetilde{\mathbf L}_{m,\tau}\ \ge\ \sup_{m}g^c_{m,\tau}\ =\ \mathbf C_{\ell, \tau},\mbox{ and }\\
\mathbf C_{a,\tau}&\ =\ \sup_{m} \widetilde{\mathbf L}_{m,\tau}\ \ge\ \sup_{m} \frac{\nu_{m,\tau}}{\tau}\ =\ \frac{\mathbf C_{b,\tau}}{\tau}.
\end{align*}
Hence, $\mathcal{B}$ is $(\mathbf C_{a,\tau}, \tau)$-quasi-greedy and has $\tau \mathbf C_{\alpha, \tau}$-property (A, $\tau$).

(2) By \eqref{e2}, if $\mathcal{B}$ is $(\mathbf C_{\ell,\tau},\tau)$-quasi-greedy and has $\mathbf C_{b,\tau}$-property (A, $\tau$), 
\begin{align*}
    \sup_{m\ge 1}\tilde{\mathbf L}_{m,\tau} \ \le\ \sup_{m\ge 1}\frac{g^c_{m-1,\tau}\nu_{m,\tau}}{\tau}\ \le\ \frac{1}{\tau}\sup_{m\ge 1}g^c_{m-1,\tau}\sup_{m\ge 1}\nu_{m,\tau}\ \le\ \frac{\mathbf C_{\ell,\tau}\mathbf C_{b,\tau}}{\tau}.
\end{align*}
Therefore, $\mathcal{B}$ is $(\frac{\mathbf C_{\ell, \tau}\mathbf C_{b,\tau}}{\tau},\tau)$-almost greedy.
\end{proof}

\subsection{Estimates on the $\tau$-partially greedy constant (Generalization of \cite[Theorem 3.11]{K})}
\begin{defi}\normalfont
A basis $\mathcal{B}$ is $(C,\tau)$-partially greedy if $C:=C(\tau)\ge 1$ verifies
\begin{equation*}\label{e33}
    \sup_{m}\mathbf L^{re}_{m,\tau} \ \le\ C\ <\ \infty.
\end{equation*}
In this case, we let $\mathbf C_{p,\tau}:= \sup_{m}\mathbf L^{re}_{m,\tau}$.

A basis $\mathcal{B}$ is $(C,\tau)$-strong partially greedy if $C:=C(\tau)\ge 1$ verifies 
\begin{equation*}
    \sup_{m}\widehat{\mathbf L}^{re}_{m,\tau} \ \le\ C \ <\ \infty.
\end{equation*}
We let $\mathbf C_{sp,\tau}:= \sup_{m}\widehat{\mathbf L}^{re}_{m,\tau}$.
\end{defi}

The classical definition of partially greedy bases was due to Dilworth et al. \cite{DKKT}, where they characterized such a basis as being quasi-greedy and conservative (defined later.)

\begin{defi}\normalfont
A basis $\mathcal{B}$ is said to have $C$-left property (A, $\tau$) if $C:=C(\tau)\ge 1$ verifies
\begin{equation*}\label{e32}
    \sup_m \nu_{m,\tau,\ell}\ \le \ C\ <\ \infty.
\end{equation*}
In this case, we let $\mathbf C_{L,\tau}:= \sup_{m}\nu_{m,\tau,\ell}$. 
\end{defi}

\begin{defi}\normalfont
A basis $\mathcal{B}$ is $C$-super-conservative if 
$$\sup_{m\ge 1} \psi_m\ \le\ C\ <\ \infty.$$
We let $\mathbf C_{sc}:= \sup_{m\ge 1}\psi_m$.
\end{defi}

\begin{defi}\normalfont
A basis $\mathcal{B}$ is said to be $(C, \tau)$-partially symmetric for largest coefficients if $C:= C(\tau)\ge 1$ verifies
\begin{equation*}
    \sup_m \nu'_{m,\tau,\ell}\ \le\ C\ <\ \infty.
\end{equation*}
In this case, we let $\mathbf C_{pl, \tau}: = \sup_m\nu'_{m,\tau, \ell}$.
\end{defi}

In \cite{K}, Khurana characterized partially greedy Schauder bases by quasi-greediness and left-property (A), while in \cite{B0}, Bern\'{a} characterized strong partially greedy M-bases by partial symmetry for largest coefficients and quasi-greediness. We now generalize these results to the setting of $\tau$-weak greedy sets.

\begin{thm}
Let $\mathcal{B}$ be an M-basis. 
\begin{enumerate}
    \item If $\mathcal{B}$ is $(\mathbf C_{sp,\tau},\tau)$-strong partially greedy, then $\mathcal{B}$ is $(\mathbf C_{sp,\tau},\tau)$-quasi-greedy and is $(\tau \mathbf C_{sp,\tau}, \tau)$-partially symmetric for largest coefficients.
    \item Conversely, if $\mathcal{B}$ is $(\mathbf C_{\ell, \tau},\tau)$-quasi-greedy and is $(\mathbf C_{pl,\tau},\tau)$-partially symmetric for largest coefficients, then $\mathcal{B}$ is $(\frac{\mathbf C_{\ell,\tau}\mathbf C_{pl,\tau}}{\tau},\tau)$-strong partially greedy. 
\end{enumerate}
\end{thm}

\begin{proof}
(1) By \eqref{e18} and \eqref{e10}, if $\mathcal B$ is $(\mathbf C_{sp,\tau},\tau)$-strong partially greedy, 
\begin{align*}\sup_{m} g_{m,\tau}^c&\ \le\ \sup_{m} \widehat{\mathbf L}^{re}_{m,\tau}\ =\ \mathbf C_{sp,\tau},\\
\sup_{m}\frac{\nu'_{m,\tau,\ell}}{\tau} &\ \le\ \sup_{m}\widehat{\mathbf L}^{re}_{m,\tau}\ =\ \mathbf C_{sp,\tau}.
\end{align*}
So, $\mathcal{B}$ is $(\mathbf C_{sp,\tau},\tau)$-quasi-greedy and $(\tau C_{sp, \tau},\tau)$-partially symmetric for largest coefficients.

(2) This is immediate from the upper bound of $\widehat{\mathbf L}^{re}_{m,\tau}$ in \eqref{e10}. 
\end{proof}

\begin{thm}\label{m9}
Let $\mathcal{B}$ be an M-basis.
\begin{enumerate}
    \item If $\mathcal{B}$ is Schauder and $(\mathbf C_{p,\tau},\tau)$-partially greedy, then $\mathcal{B}$ is $((\mathbf K_b+1)\mathbf C_{p,\tau}, \tau)$-quasi-greedy and has $\mathbf C_{L,\tau}$-left property (A, $\tau$) with
    \begin{equation}\label{e41}\mathbf C_{L,\tau}\ \le\ \tau\mathbf C_{p,\tau}(2+\mathbf K_b + \mathbf C_{p,\tau} + \mathbf C_{p,\tau}\mathbf K_b).\end{equation}
    \item Conversely, if $\mathcal{B}$ is $(\mathbf C_{\ell,\tau},\tau)$-quasi-greedy and has $(\mathbf C_{L,\tau},\tau)$-left property (A, $\tau$), then $\mathcal{B}$ is $(\frac{\mathbf C_{\ell,\tau} \mathbf C_{L,\tau}}{\tau},\tau)$-partially greedy. 
\end{enumerate}
\end{thm}

Before proving Theorem \ref{m9}, we need the following lemma.

\begin{lem}\label{l20}
Let $\mathcal{B}$ be $(\mathbf C_{\ell,\tau}, \tau)$-quasi-greedy. Then $\mathcal{B}$ is super-conservative if and only if it has left property (A, $\tau$). In particular, 
\begin{equation*}\mathbf C_{sc}\ \le\ \mathbf C_{L,\tau}\ \le\ \tau \mathbf C_{\ell, \tau}+\mathbf C_{sc} + \mathbf C_{\ell, \tau}\mathbf C_{sc}.
\end{equation*}
\end{lem}

\begin{proof}Since $\nu_{m,\tau,\ell}\ge\psi_m$ for all $m\ge 1$, if $\mathcal B$ has left property (A), then it is super-conservative, and the left inequality is obvious.
Assume that $\mathcal{B}$ is $\mathbf C_{sc}$-super-conservative. Let $x, A, B, (\varepsilon), (\delta)$ be chosen as in \eqref{e12}. We have
\begin{align*}\|\tau x+ 1_{\delta B}\|&\ \le \ \tau\|x\| + \|1_{\delta B}\|
\ \le\ \tau \mathbf C_{\ell, \tau}\|x+1_{\varepsilon A}\| + \mathbf C_{sc}\|1_{\varepsilon A}\|\\
&\ \le\ (\tau \mathbf C_{\ell, \tau}+\mathbf C_{sc}(\mathbf C_{\ell,\tau}+1))\|x+1_{\varepsilon A}\|.
\end{align*}
Hence, $\mathbf C_{L,\tau}\le \tau \mathbf C_{\ell,\tau} + \mathbf C_{sc} + \mathbf C_{\ell,\tau}\mathbf C_{sc}$.
\end{proof}

\begin{proof}[Proof of Theorem \ref{m9}]
(1)  Assume that $\mathcal{B}$ is $(\mathbf C_{p, \tau},\tau)$-partially greedy. From \eqref{e15}, we know that $\mathcal B$ is $(\mathbf C_{p,\tau}(\mathbf K_b + 1), \tau)$-quasi-greedy. 

Fix $A, B, (\varepsilon), (\delta)$ as in \eqref{e22}. Let $y: = 1_{\delta B} + 1_D+ \tau 1_{\varepsilon A}$, where $D = \{1,\ldots, \max B\}\backslash B$,  then $D\cup A\in G(y, \max B, \tau)$. 
We have 
\begin{align*}
    \|1_{\delta B}\| \ =\ \|y - P_{D\cup A}(y)\| \ \le\ \mathbf C_{p,\tau}\|y-S_{\max B}(y)\| \ =\ \mathbf C_{p,\tau}\tau \|1_{\varepsilon A}\|.
\end{align*}
Hence, $\mathcal{B}$ is $(\mathbf C_{p,\tau}\tau)$-super-conservative. By Lemma \ref{l20} and due to $(\mathbf C_{p,\tau}(\mathbf K_b+1),\tau)$-quasi-greediness, we obtain \eqref{e41}. 

(2) This is immediate from the upper bound of $\mathbf L^{re}_{m,\tau}$ in \eqref{e3}.
\end{proof}

\section{Property (A, $\tau$), uniform property (A), and quasi-greedy bases}\label{uniform}
Let us retrieve some classical definitions when $\tau = 1$ in the definitions given in Section \ref{sd}:
\begin{defi}\label{d30}\normalfont\mbox{ }
\begin{itemize} 
    \item \textit{Property (A)}: If $\mathcal{B}$ has $\mathbf{C}_{b,1}$-property (A, $1$), we say $\mathcal{B}$ has $\mathbf C_b$-property (A) (here $\mathbf C_b := \mathbf C_{b,1}$) or property (A) if we do not want to specify $\mathbf C_b$.
    \item \textit{Greedy}: If $\mathcal{B}$ is $(\mathbf C_{g,1},1)$-greedy, we say $\mathcal B$ is $\mathbf C_g$-greedy (here $\mathbf C_g:= \mathbf C_{g,1}$.) We may also say $\mathcal B$ is greedy without specifying $\mathbf C_g$.
    \item \textit{Almost greedy}: If $\mathcal{B}$ is $(\mathbf C_{a,1},1)$-almost greedy, we say $\mathcal B$ is $\mathbf C_a$-almost greedy (here $\mathbf C_a:= \mathbf C_{a,1}$) or is almost greedy without specifying the constant. 
    \item \textit{Quasi-greedy}\footnote{Wojtaszczyk \cite{W} claimed that $\sup_{m} \|G_m\| < \infty$ is equivalent to the convergence of greedy algorithms $(G_m(x))_{m=1}^\infty$ to $x$ for each $x\in\mathbb{X}$. Such an equivalence holds if and only if our M-basis is strong (see \cite[Remark 6.2]{BBGHO}.)}: If $\mathcal{B}$ is $(\mathbf C_{\ell,1}, 1)$-quasi-greedy, we say $\mathcal B$ is $\mathbf C_\ell$-suppression quasi-greedy (here $\mathbf C_\ell:= \mathbf C_{\ell,1}$) (or simply, is quasi-greedy). In this case, we may also say $\mathcal B$ is quasi-greedy without specifying the constant. While $\mathbf C_\ell = \sup_m g_{m,1}^c$, we let $\mathbf C_w := \sup_m g_{m,1}$ and say $\mathcal{B}$ is $\mathbf C_w$-quasi-greedy. Clearly, $\mathbf C_\ell - 1\le \mathbf C_w\le \mathbf C_\ell + 1$. 
    \item \textit{Super-democratic}: $\mathcal B$ is said to be $C$-super-democratic if $\sup_m \mu_m \le C <\infty$.  In this case, we let $\mathbf C_{sd}:= \sup_{m\ge 1} \mu_m$.
\end{itemize}
\end{defi}

We now study the relation among property $(A, \tau)$ for different values of $\tau$.

\subsection{Truncation operator}
For each $\alpha > 0$, we define the truncation function $T_\alpha$ as follows: for  $b\in\mathbb{F}$,
$$T_{\alpha}(b)\ =\ \begin{cases}\sgn(b)\alpha, &\mbox{ if }|b| > \alpha,\\ b, &\mbox{ if }|b|\le \alpha.\end{cases}$$
We define the truncation operator $T_\alpha: \mathbb{X}\rightarrow \mathbb{X}$ as
$$T_{\alpha}(x)\ =\ \sum_{n=1}^\infty T_\alpha(e_n^*(x))e_n \ =\ \alpha 1_{\varepsilon \Gamma_{\alpha}(x)}+ P_{\Gamma_\alpha^c(x)}(x),$$
where $\Gamma_\alpha(x) = \{n: |e_n^*(x)| > \alpha\}$ and $\varepsilon_n = \sgn(e_n^*(x))$ for all $n\in \Gamma_\alpha(x)$.

\begin{thm}\label{bto}\cite[Lemma 2.5]{BBG} Let $\mathcal{B}$ be $\mathbf{C}_\ell$-suppression quasi-greedy. Then for any $\alpha > 0$, $\|T_\alpha\|\le \mathbf{C}_\ell$.
\end{thm}

\begin{proof}
Simply observe that this lemma is a special of Proposition \ref{p2} when $\tau = 1$.
\end{proof}

\subsection{Relation among property (A, $\tau$) for different $\tau$}

\begin{prop}\label{pp1}Let $\mathcal B$ be an M-basis. The following hold. 
\begin{enumerate}
\item If $\mathcal B$ is $\mathbf{K}_s$-suppression unconditional and has $\mathbf C_{b}$-property (A), then $\mathcal B$ has $\mathbf C_b\mathbf{K}_s$-property (A, $\tau$) for all $\tau\in (0,1]$.
\item Let $0 < \tau_2 < \tau_1 \le 1$. If $\mathcal B$ has $\mathbf C_{b,\tau_2}$-property (A, $\tau_2$), then $\mathcal B$ has 
$\mathbf C_{b,\tau_2}\tau_1/\tau_2$-property (A, $\tau_1$). 
\end{enumerate}
\end{prop}

\begin{proof}
We prove (1). Assume that $\mathcal B$ is $\mathbf{K}_s$-suppression unconditional and has $\mathbf C_b$-property (A). Choose $A, B, x, (\varepsilon), (\delta)$ as in \eqref{e7}. We have
$$\|x+1_{\varepsilon A}\|\ \ge\ \frac{1}{\mathbf{K}_s}\|\tau x+1_{\varepsilon A}\|\ \ge\ \frac{1}{\mathbf C_b\mathbf{K}_s}\|\tau x + 1_{\delta B}\|.$$
The first inequality is due to Proposition \ref{p1}, while the second inequality is due to $\mathbf C_b$-property (A). 

Next, we prove (2). Assume that $\mathcal B$ has $\mathbf C_{b,\tau_2}$-property (A, $\tau_2$). Let us choose $A, B, x, (\varepsilon), (\delta)$ as in \eqref{e7}. Here $\|x\|_\infty\le 1/\tau_1$. By $\mathbf C_{b,\tau_2}$-property (A, $\tau_2$), we have
$$\|x+1_{\varepsilon A}\|\ \ge\ \frac{1}{\mathbf C_{b, \tau_2}}\|\tau_2 x+1_{\theta B}\|, \forall \mbox{ signs } (\theta).$$
Hence, 
$$\|x+1_{\varepsilon A}\|\ \ge\ \frac{\tau_2}{\tau_1\mathbf C_{b,\tau_2}}\sup_{(\theta)}\left\|\tau_1 x + \frac{\tau_1}{\tau_2}1_{\theta B}\right\|\ \ge\ \frac{\tau_2}{\tau_1 \mathbf C_{b, \tau_2}}\|\tau_1x+1_{\delta B}\|,$$
where the last inequality is due to norm convexity. 
\end{proof}

In going from property (A) to property (A, $\tau$) in the above proposition, we need unconditionality to ``flatten out" the vector $x$ so that $\|x\|_\infty\le 1$, which makes it possible to apply property (A). Like unconditionality, quasi-greediness can obtain the same result thanks to the boundedness of the truncation operator $T_1$ (see Theorem \ref{bto}.) We can therefore relax unconditionality in Proposition \ref{pp1} as follows:

\begin{prop}\label{pp2}
Let $\tau\in (0,1]$. If an M-basis $\mathcal B$ is $\mathbf C_\ell$-suppression quasi-greedy and has $\mathbf C_{b,\tau}$-property (A, $\tau$), then $\mathcal B$ has $2\mathbf C_{b,\tau}^2\mathbf C_\ell$-property (A, $\gamma$) for all $\gamma\in (0,1]$.
\end{prop}

\begin{proof}
Assume that our basis $\mathcal{B}$ has $\mathbf C_{b, \tau}$-property (A, $\tau$) and is $\mathbf C_\ell$-suppression quasi-greedy. Fix $\gamma\in (0,1]$. Let $x, A, B, (\varepsilon), (\delta)$ be chosen as in \eqref{e7}, where $x\in \mathbb{X}_c$ and $\|x\|_\infty\le 1/\gamma$. Let $E = \{n: |e_n^*(x)|\ge 1\}$ and $T_1$ be the $1$-truncation operator. 
We shall show that
\begin{equation}\label{99}\|x+1_{\varepsilon A}\| \ \ge\ \frac{1}{2\mathbf C_\ell\mathbf C^2_{b,\tau}}\|\gamma x+1_{\delta B}\|.\end{equation}

Case 1: $\gamma\neq \tau^2$.
Write
\begin{equation}\label{100}
    \|x+1_{\varepsilon A}\|\ =\ \|P_Ex + 1_{\varepsilon A} + P_E^cx\|\mbox{ and }\|\gamma x+1_{\delta B}\|\ =\ \|\gamma P_Ex + \gamma P_E^cx + 1_{\delta B}\|.
\end{equation}
Since $P_Ex + 1_{\varepsilon A}$ is a greedy sum of $x+1_{\varepsilon A}$, we have
\begin{equation}\label{101}
    \|x+1_{\varepsilon A}\|\ \ge\ \frac{1}{\mathbf C_\ell}\|P_E^cx\|\ =\ \frac{1}{|\gamma-\tau^2| \mathbf C_\ell}\| (\gamma-\tau^2)P_E^cx\|\ \ge\ \frac{1}{ \mathbf C_\ell}\| (\gamma-\tau^2)P_E^cx\|.
\end{equation}
Choose a sign $(\psi)$ such that $\psi_n = \sgn(e_n^*(x))$ for all $n\in \supp(x)$. Choose $D\subset \mathbb{N}$ with $\max (A\cup B \cup \supp(x)) < \min D$ and $|D| = |A| + |E|$. By Theorem \ref{bto} and $\mathbf C_{b,\tau}$-property (A, $\tau$), we obtain
\begin{align}\label{102}
    \|x+1_{\varepsilon A}\|&\ \ge\ \frac{1}{\mathbf C_\ell}\|1_{\psi E} + 1_{\varepsilon A} + P_E^cx\|\ \ge\ \frac{1}{\mathbf C_\ell \mathbf C_{b,\tau}}\|\tau P_E^cx + 1_D\|\nonumber\\
    &\ \ge\ \frac{1}{\mathbf C_\ell \mathbf C_{b,\tau}^2}\sup_{(\theta)}\|\tau^2 P_E^cx + 1_{\theta E} +  1_{\delta B}\|\nonumber\\
     &\ \ge\ \frac{1}{\mathbf C_\ell \mathbf C_{b,\tau}^2}\|\tau^2 P_E^cx + \gamma P_Ex+  1_{\delta B}\|.
\end{align}
where the last inequality is due to norm convexity. It follows from \eqref{100}, \eqref{101}, and \eqref{102} that
\begin{align*}
    \|x+1_{\varepsilon A}\| &\ \ge\ \frac{1}{2\mathbf C_\ell}\|(\gamma-\tau^2) P_E^cx\| + \frac{1}{2\mathbf C_\ell \mathbf C^2_{b,\tau}}\|\tau^2 P_E^cx+1_{\delta B} + \gamma P_E x\|\\
    &\ \ge\ \frac{1}{2\mathbf C_\ell \mathbf C^2_{b,\tau}}\|\gamma P_E^cx + \gamma P_E x + 1_{\delta B}\|\\
    &\ =\ \frac{1}{2\mathbf C_\ell \mathbf C^2_{b,\tau}}\|\gamma x + 1_{\delta B}\|,
\end{align*}
which is \eqref{99}, as desired. 

Case 2: $\gamma\neq \tau^2$. Simply use \eqref{102} to get $\|x+1_{\varepsilon A}\|\ge \frac{1}{\mathbf C_\ell \mathbf C_{b,\tau}^2}\|\gamma x+  1_{\delta B}\|$.
\end{proof}

\begin{cor}\label{pc10}
If an M-basis $\mathcal B$ has $1$-property (A), then for all $\tau\in (0,1)$, $\mathcal B$ has $2$-property (A, $\tau$). 
\end{cor}

\begin{proof}
Let $\tau\in (0,1)$. Assume $\mathcal B$ has $1$-property (A). By \cite[Proposition 2.5]{AA}, $\mathcal B$ is $1$-suppresion quasi-greedy.  Plugging in $\mathbf C_\ell = \tau_1 = \mathbf C_{b,\tau_1} = 1$ into Proposition \ref{pp2}, we know that $\mathcal B$ has $2$-property (A, $\tau$).
\end{proof}

Proposition \ref{pp1} item (2) suggests that we may need to increase the property (A, $\tau$) constant when $\tau$ increases from $\tau_2$ to $\tau_1$. On the other hand, an interesting feature of Proposition \ref{pp1} item (1), Proposition \ref{pp2}, and Corollary \ref{pc10} is that there exists a constant $C$ (independent of $\tau$) such that $\mathcal B$ has $C$-property (A, $\tau$) for all $\tau\in (0,1]$. Such a basis is said to have \textit{uniform property (A)}.

The following corollary is immediate from Proposition \ref{pp2}. 

\begin{cor}\label{ptu}
Fix $\tau\in (0,1]$. If an M-basis $\mathcal B$ has property (A, $\tau$) and is quasi-greedy, then $\mathcal B$ has uniform property (A). 
\end{cor}

A natural question is whether the converse of Corollary \ref{ptu} holds; that is, whether uniform property (A) implies quasi-greediness. We know property (A) alone does not imply quasi-greediness (see \cite[Proposition 5.7]{BBG}.) Below we construct a basis that has uniform property (A) but is not quasi-greedy. Our construction is somewhat simpler than \cite[Proposition 5.7]{BBG} since we do not attempt to optimize various parameters.

\begin{thm}\label{nu}
There exists a Banach space $\mathbb{X}$ with a monotone ($\mathbf K_b = 1$) Schauder basis that is non-quasi-greedy and has uniform property (A). 
\end{thm}

The next subsection establishes Theorem \ref{nu}.

\subsection{Non-quasi-greedy basis with uniform property (A)}

\begin{prop}\label{ppp1}
Let $\mathbb{X}$ be the completion of $c_{00}$ with respect to the norm $\|\cdot\| = \max \{\|\cdot\|_1, \|\cdot\|_{\ell_2}\}$, where $\|\cdot\|_1$ is a semi-norm on $c_{00}$. Assume $(\mathbb X, \|\cdot\|)$ has a basis $\mathcal B$. If
there exists $\lambda \ge 1$ such that for all signs $(\delta)$ and nonempty, finite sets $A\subset \mathbb{N}$, we have  $\|1_{\delta A}\|_1 \le \lambda |A|^{1/2}$, then $\mathcal B$ has $(3\lambda)$-property (A, $\tau$) for all $\tau\in (0,1]$ (with respect to $\|\cdot\|$.)
\end{prop}

\begin{proof}
Let $x\in c_{00}$ and $A, B\subset \mathbb{N}$ be finite with $|A| = |B| = n$ and $A\sqcup B\sqcup x$. Pick signs $(\varepsilon), (\delta)$. We want to show that 
$\|\tau x+1_{\delta B}\|\ \le\ 3\lambda \|x+1_{\varepsilon A}\|$.
Assume, for a contradiction,  that $\|\tau x+1_{\delta B}\|\ >\ 3\lambda \|x+1_{\varepsilon A}\|$. Since $\|\tau x+1_{\delta B}\|_{\ell_2} \le \|x+1_{\varepsilon A}\|_{\ell_2}$, it must be that $\|\tau x+1_{\delta B}\| = \|\tau x+1_{\delta B}\|_1$. 
We have
\begin{align*}
    \|x+1_{\varepsilon A}\|_1 &\ \ge\ \|x\|_1 - \|1_{\varepsilon A}\|_1\ \ge\ \|x\|_1 - \lambda\sqrt{n}\\
    &\ \ge\  \|\tau x\|_1 + \|1_{\delta B}\|_1 - 2\lambda\sqrt{n} \ \ge\ \|\tau x+1_{\delta B}\| - 2\lambda\sqrt{n}.
\end{align*}
Therefore, 
$$\|\tau x+1_{\delta B}\| \ =\ \|\tau x+1_{\delta B}\|_1 \ >\ 3\|x+1_{\varepsilon A}\|_1\ \ge\ 3(\|\tau x+1_{\delta B}\| - 2\lambda\sqrt{n}),$$
which gives \begin{equation}\label{e34}\|\tau x+1_{\delta B}\| < 3\lambda \sqrt{n}.\end{equation}
On the other hand, by assumption, 
$$\|\tau x+1_{\delta B}\|\ >\ 3\lambda\|x+1_{\varepsilon A}\|_{\ell_2} \ \ge\ 3\lambda\sqrt{n},$$
contradicting \eqref{e34}. 
\end{proof}

\subsubsection{Construct the weight sequence}
For $n\ge 1$, let $t_n = \frac{1}{\sqrt{n}}$, $L_n = e^{(\log n)^2}$ and $a_n = \frac{1}{\sqrt{n}\log (n+1)}$. Define an increasing sequence $(N_n)_{n=0}^\infty$ recursively. Set $N_0 = 0$. Choose $N_1>10$ to be the smallest such that 
$$b_1 \ :=\ a_1t_1\left(\sum_{n=1}^{N_1} t_n\right)^{-1} \ <\ \frac{a_1}{L_1}.$$
Once $N_j$ and $b_j$ are defined for $j\ge 1$, choose $N_{j+1}$ to be the smallest such that 
\begin{equation}\label{ee2}b_{j+1}\ :=\ a_{j+1}t_{j+1}\left(\sum_{n=N_j+1}^{N_j+ N_{j+1}} t_n\right)^{-1}\ <\ \min\left\{\frac{a_{j+1}}{L_{j+1}}, b_j\right\}\mbox{ and }\frac{N_{j+1}}{N_j} \ >\ 10.\end{equation}
For $j\ge 1$, denote the finite sequence $(t_{n})_{n= N_{j-1}+1}^{N_{j-1}+N_j} = (t_{N_{j-1}+1}, \ldots, t_{N_{j-1}+N_j})$ by $B_j$.

We now define weight $\omega:=(w_n)_{n=1}^\infty$ on $\mathbb{N}$ as follows:
$$(w_n) \ =\ (t_1, B_1, t_2, B_2, t_3, B_4, t_4, \ldots).$$
In words, the weight $(w_n)$ is chosen such that the first one is $t_1$; the next $N_1$ weights are taken from $B_1$ in the same order; the next weight is $t_2$, then the next $N_2$ weights are taken from $B_2$ in the same order, and so on.
\subsubsection{Construct the desired space and verify its properties}
Let $\mathbb{X}$ be the completion of $c_{00}$ under the following norm 
$$\left\|\sum_{n}x_n e_n\right\|\ =\ \max\left\{\sup_{N}\left|\sum_{n=N}^\infty w_nx_n\right|, \left(\sum_{n=1}^\infty |x_n|^2\right)^{1/2}\right\}.$$
For ease of notation, set 
$$\left\|\sum_{n}x_n e_n\right\|_1\: = \ \sup_{N}\left|\sum_{n=N}^\infty w_nx_n\right|\mbox{ and }\left\|\sum_{n}x_n e_n\right\|_2\ =\ \left(\sum_{n=1}^\infty |x_n|^2\right)^{1/2}.$$
Clearly, the standard unit vector basis $\mathcal B$ is Schauder and normalized in $(\mathbb{X}, \|\cdot\|)$. 

\begin{prop}
The basis $\mathcal B$ has uniform property (A). 
\end{prop}
\begin{proof}
It suffices to verify that the condition in Proposition \ref{ppp1} is satisfied. Choose $A\subset\mathbb{N}$ with $|A| = m$ and sign $(\delta)$. By construction, each value of $t_n$ appears exactly twice in $\omega$. Hence, 
$$\|1_{\delta A}\|_1 \ <\ 2\sum_{n=1}^m t_n\ =\ 2+ 2\sum_{n=2}^m \frac{1}{\sqrt{n}} \ \le\ 2 + 2\int_{1}^m \frac{dx}{\sqrt{x}}\ =\ 2\sqrt{m},$$
as desired. 
\end{proof}

\begin{prop}\label{t1}
The basis $\mathcal B$ is not quasi-greedy. 
\end{prop}

\begin{lem}\label{ll1}
The following vector is in $\mathbb{X}$:
$$x\ =\ (a_1, \underbrace{-b_1, \ldots, -b_1}_{N_1 \text{times}}, a_2, \underbrace{-b_2, \ldots, -b_2}_{N_2 \text{times}}, \ldots , a_j, \underbrace{-b_j, \ldots, -b_j}_{N_j \text{times}}, \ldots).$$
\end{lem}

\begin{proof}
Let $S_N$ be the partial sum of order $N$ and $T_N = I - S_N$ be the $N$th tail operator. Fix $\varepsilon > 0$. Call $(a_j, \underbrace{-b_j, \ldots, -b_j}_{N_j \text{times}})$ the $j$th block. Pick $k$ such that $a_{k}t_{k} = \frac{1}{k\log(k+1)} < \varepsilon$. Let $N_0$ be sufficiently large such that $\min\supp (T_{N_0}(x))$ does not lie before the $k$th block. 
By construction of $(w_n)$, $(a_n)$, and $(b_n)$, if a block fully appears in a tail, then it contributes nothing to the norm $\|\cdot\|_1$. Hence, for all $N\ge N_0$, $\|T_Nx\|_1\le a_kt_k < \varepsilon$, and we have $S_N(x) \rightarrow x$ in $\|\cdot\|_1$.

Next, we show that $\|x\|_2 < \infty$. Observe that
$$\sum_{n=1}^\infty (a_n)^2 \ =\ \sum_{n=1}^\infty \frac{1}{n(\log (n+1))^2} \ < \ 4.$$
We also have
\begin{align*}
    \sum_{n=1}^\infty N_n (b_n)^2&\ =\ \sum_{n=1}^\infty N_na_n^2t_n^2\left(\sum_{j=N_{n-1}+1}^{N_{n-1}+N_n}\frac{1}{\sqrt{j}}\right)^{-2}\ \le\ \sum_{n=1}^\infty N_na_n^2t_n^2\left(\int_{N_{n-1}+1}^{N_{n-1}+N_n+1}\frac{dx}{\sqrt{x}}\right)^{-2}\\
    &\ \le\ \sum_{n=1}^\infty \frac{1}{n^2\log^2 (n+1)(\sqrt{N_{n-1}/N_n+1+1/N_n}-\sqrt{N_{n-1}/N_n+1/N_n})^2}\\
    &\ \le\ 3\sum_{n=1}^\infty \frac{1}{n^2\log^2(n+1)},
\end{align*}
where the last inequality is due to $1/N_n+N_{n-1}/N_n < 0.2$. 
We have shown that $\|x\|_2 < \infty$ and so, $x\in \mathbb{X}$. 
\end{proof}

\begin{proof}[Proof of Proposition \ref{t1}]
Consider the vector $x$ as in Lemma \ref{ll1}:
$$x\ =\ (a_1, \underbrace{-b_1, \ldots, -b_1}_{N_1 \text{times}}, a_2, \underbrace{-b_2, \ldots, -b_2}_{N_2 \text{times}}, \ldots , a_j, \underbrace{-b_j, \ldots, -b_j}_{N_j \text{times}}, \ldots).$$
Let $k\in \mathbb{N}$. Pick $\varepsilon\in (b_{k+1}, b_k)$ and consider $T_\varepsilon(x)$, where $T_\varepsilon$ is the thresholding operator. In particular, for each $x = \sum_{n}x_ne_n\in \mathbb{X}$, 
$$T_\varepsilon (x) \ :=\ \sum_{n: |x_n|>\varepsilon}x_ne_n.$$ 
Since $(b_n)$ is strictly decreasing, $T_\varepsilon (x)$ does not contain any coefficient $b_n$ for $n\ge k+1$; that is, from the $(k+1)$th block onward, only certain values of $a_n$ at the start of the $n$th block appear in $T_\varepsilon(x)$. Specifically, let us establish a range of $n$ such that $a_n$ still appears in the $n$th block. By construction, 
$$\varepsilon \ <\ b_k \ <\ \frac{a_k}{L_k}\ =\ \frac{1}{(\log (k+1))\sqrt{k}e^{(\log (k))^2}}.$$
We want to find $n$ such that
$$a_n \ >\ \frac{1}{(\log k)\sqrt{k}e^{(\log (k))^2}}\ >\ \varepsilon.$$
Equivalently, 
$$\frac{1}{\sqrt{n}\log (n+1)}\ >\ \frac{1}{(\log k)\sqrt{k}e^{(\log (k))^2}}.$$
It is easy to check that for sufficiently large $k$, we have $n \le \frac{1}{2}e^{(\log(k))^2/2}$ work. Therefore,
\begin{align*}\|T_\varepsilon(x)\|_1 \ \ge\ \sum_{n=k+1}^{\frac{1}{2}e^{(\log(k))^2/2}}t_na_n&\ =\ \sum_{n=k+1}^{\frac{1}{2}e^{(\log(k))^2/2}}\frac{1}{n\log (n+1)} \ge\  \frac{1}{2}\int_{k+1}^{\frac{1}{2}e^{(\log(k))^2/2}}\frac{dx}{x\log x}\\
&\ =\ \frac{1}{2} \left(\log\log \left(\frac{1}{2}e^{(\log(k))^2/2}\right)-\log\log (k+1)\right)\rightarrow\infty
\end{align*}
as $k\rightarrow\infty$. Therefore, $\mathcal B$ is not quasi-greedy. 
\end{proof}

\subsection{Yet another characterization of property (A)} Lemma \ref{l1} gives a characterization of property (A, $\tau$) for $M$-bases. We now give another characterization of property (A, $\tau$), which in turn gives a seemingly new characterization of property (A). 

\begin{lem}\label{pl7}
If a basis $\mathcal B$ has $\mathbf{C}_{b,\tau}$-property $(A, \tau)$, then $\mathcal B$ is $\min\{2\mathbf{C}_{b,\tau}/\tau, \mathbf{C}_{b,\tau}^2\}$-super-democratic. 
\end{lem}

\begin{proof}
By Proposition \ref{p1} item (2), $\mathcal B$ has $\mathbf{C}_{b,\tau}/\tau$-property (A). Let $A, B, (\varepsilon), (\delta)$ be chosen as in \eqref{e50}. Without loss of generality, we can assume that $\varepsilon \equiv 1$. By \cite[Lemma 6.4]{DKO}, we know that $1_{\delta B}\in 2S$, where $S = \{\sum_{D\subset B}a_D1_{D}: (a_D)\subset \mathbb{R}, \sum_{D\subset B}|a_D|\le 1\}$. We shall show that
$$\|1_D\|\ \le\ \frac{\mathbf{C}_{b,\tau}}{\tau}\|1_A\|, \forall D\subset B.$$
We have
\begin{align*}
    \|1_D\|\ =\ \|1_{D\backslash A} + 1_{D\cap A}\|\ \le\ \frac{\mathbf{C}_{b,\tau}}{\tau}\|1_{A\backslash D} + 1_{D\cap A}\|\ =\ \frac{\mathbf{C}_{b,\tau}}{\tau}\|1_A\|.
\end{align*}
The inequality is due to $\mathbf{C}_{b,\tau}/\tau$-property (A) and the fact that $|D\backslash A|\le |A\backslash D|$. Therefore, $\mathcal B$ is $2\mathbf{C}_{b,\tau}/\tau$-super-democratic. 

To see that $\mathcal B$ is $\mathbf{C}_{b,\tau}^2$-super-democratic, pick any set $E\subset\mathbb{N}$ with $|E| = |A| = |B|$ and $ E > A\cup B$.
By $\mathbf{C}_{b,\tau}$-property (A, $\tau$), we know that $\|1_{\delta B}\|\le \mathbf{C}_{b,\tau}\|1_{E}\|$ and $\|1_E\|\le \mathbf{C}_{b,\tau}\|1_{\varepsilon A}\|$. Hence,
$$\|1_{\delta B}\|\le \mathbf{C}_{b,\tau}^2\|1_{\varepsilon A}\|.$$
\end{proof}

\begin{thm}\label{charproperty(A)tau}
For an M-basis $\mathcal B$, the following are equivalent:
\begin{enumerate}
    \item $\mathcal B$ has property (A, $\tau$).
    \item There exists a function $f:\mathbb{N}\rightarrow \mathbb{R}_+$ and constants $c_1, c_2>0$ such that 
    \begin{equation}\label{pe4}\|1_{\varepsilon \Lambda}\|\le c_1f(|\Lambda|)\mbox{ and }
     \|1_{\varepsilon \Lambda}+x\|\ge c_2f(|\Lambda|),\end{equation}
for any $x\in\mathbb{X}$ with $\|x\|_{\infty}\le 1/\tau$, any nonempty, finite set $\Lambda\subset\mathbb{N}$ with $\supp(x)\cap \Lambda = \emptyset$, and sign $(\varepsilon)$. Furthermore, $f$ is unique in the following sense: if $g$ is another function verifying \eqref{pe4}, then $0< \inf g(n)/f(n) \le \sup g(n)/f(n) < \infty$. 
\end{enumerate}
\end{thm}
\begin{proof}

(1) $\Rightarrow$ (2): Assume that $\mathcal B$ has $\mathbf C_{b,\tau}$-property (A, $\tau$). Define $f: \mathbb{N}\rightarrow \mathbb{R}_+$ as follows: 
$$f(n)\ :=\ \sup_{(\varepsilon), \Lambda\subset\mathbb{N}, |\Lambda|=n}\|1_{\varepsilon \Lambda}\|.$$
 We show that $f$ verifies \eqref{pe4}. Fix a nonempty, finite set $A\subset\mathbb{N}$, sign $(\varepsilon)$, and  $x\in\mathbb{X}_c$ with $\|x\|_{\infty}\le 1/\tau$ and $\supp(x)\cap \Lambda = \emptyset$. By definition, $\|1_{\varepsilon A}\|\le f(|A|)$. We have
 \begin{equation}\label{f1}\|1_{\varepsilon A}\|\ \le\ \frac{1}{2}(\|1_{\varepsilon A}+x\| + \|1_{\varepsilon A}-x\|).\end{equation}
 Choose $B > A\cup\supp(x)$ with $|B| = |A|$. Then
 \begin{equation}\label{f2}\|1_{\varepsilon A}-x\|\ \le \ \mathbf C_{b,\tau}\|1_{B}+x\|\ \le \mathbf C^2_{b,\tau}\|1_{\varepsilon A} + x\|.\end{equation}
 By \eqref{f1}, \eqref{f2}, and Lemma \ref{pl7}, we obtain
 \begin{align*}
     \|1_{\varepsilon A} + x\|\ \ge\ \frac{2}{1+\mathbf C^2_{b,\tau}}\|1_{\varepsilon A}\|\ \ge\ \frac{2}{\mathbf C^2_{b,\tau}(1+\mathbf C^2_{b,\tau})}f(|A|).
 \end{align*}
Setting $c_2 = \frac{2}{\mathbf C^2_{b,\tau}(1+\mathbf C^2_{b,\tau})}$, we are done.

(2) $\Rightarrow$ (1): Let $x, A, B, (\varepsilon), (\delta)$ be chosen as in \eqref{e7}. 

Case 1: $\|x\| \ge 2c_1f(|A|)$. Then $\|1_{\varepsilon A}\|,\|1_{\delta B}\|\le \|x\|/2$. We have
\begin{align*}
    \frac{\|\tau x+1_{\delta B}\|}{\|x+1_{\varepsilon A}\|}\ \le\ \frac{\tau\|x\|+\|1_{\delta B}\|}{\|x\|-\|1_{\varepsilon A}\|}\ \le\ \frac{(\tau + 1/2)\|x\|}{\|x\|/2}\ =\ 2\tau +1.
\end{align*}

Case 2: $\|x\| < 2c_1f(|A|)$. We have
\begin{align*}
    \frac{\|\tau x+1_{\delta B}\|}{\|x+1_{\varepsilon A}\|}\ \le\ \frac{\tau\|x\|+\|1_{\delta B}\|}{c_2f(|A|)}\ <\ \frac{2c_1\tau f(|B|)+c_1f(|B|)}{c_2f(|A|)}\ =\ \frac{(2\tau + 1)c_1}{c_2}.
\end{align*}
Therefore, $\mathcal B$ has $(2\tau + 1)c_1/c_2$-property (A, $\tau$).

Finally, we prove that if $g$ is another function verifying \eqref{pe4}, then 
$$0 \ <\ \inf \frac{f(n)}{g(n)} \ \le\ \sup \frac{f(n)}{g(n)}\ <\ \infty.$$
By symmetry, it suffices to prove that $\inf f(n)/g(n) >0$. Suppose not. Let $c_3>0$ be such that $$\|1_{\varepsilon \Lambda}+ x\|\ge c_3g(|\Lambda|),$$
for any $x\in\mathbb{X}$ with $\|x\|_{\infty}\le 1/\tau$, any finite set $\Lambda\subset\mathbb{N}$ with $\supp(x)\cap \Lambda = \emptyset$, and sign $(\varepsilon)$.
Let $N$ be chosen such that $f(N)/g(N) < c_3$. Choose $B\subset\mathbb{N}$ with $|B| = N$. By construction of $f$, we have 
$\|1_{B}\|\ \le\ f(N)$. On the other hand, $\|1_{B}\|\ \ge\ c_3g(N)$, and so, $f(N) \ge c_3g(N)$, which contradicts our choice of $N$.
\end{proof}

\section{From classical greedy-type bases to weak greedy-type bases}\label{cL}

It is obvious that a $(C,\tau)$-(quasi, almost) greedy basis is (quasi, almost) greedy (see Definition \ref{d30}.) Thanks to the work of Konyagin and Temlyakov \cite{KT}, we know that the converse is true as well; that is, for any fixed $\tau\in (0,1]$, a $C_1$-(quasi, almost) greedy basis is $(C_2, \tau)$-(quasi, almost) greedy for $C_2$ dependent on $\tau$ and $C_1$. For example, if a basis is $\mathbf C_w$-quasi-greedy, then $\mathbf C_{\ell, \tau}$ can be bounded above by a quantity involving $\mathbf C_w$ and $\tau$. Corresponding results hold for $\mathbf C_{a,\tau}$ and $\mathbf C_{g, \tau}$. In this section, we shall compute and improve these upper bounds.

\subsection{From quasi-greedy to $(\mathbf C_{\ell, \tau},\tau)$-quasi-greedy}

\begin{thm}[Konyagin and Temlyakov \cite{KT}]\label{pst2}
Let $\mathcal B$ be a $\mathbf C_w$-quasi-greedy M-basis. Then for $\tau\in (0,1]$, we have
\begin{equation}\label{pe9}\gamma_{m,\tau}(x)\ \le \ C(\mathbf{C}_w, \tau)\|x\|, \forall x\in \mathbb{X}, \forall m\in \mathbb{N}.\end{equation}
That is, $\mathcal B$ is $(C(\mathbf{C}_w, \tau), \tau)$-quasi-greedy.
\end{thm}

The constant $C(\mathbf{C}_w, \tau)$ in Theorem \ref{pst2} was shown to be of order $O(\mathbf{C}^5_w/\tau)$.
We reduce the order of $C(\mathbf{C}_w, \tau)$ to $O(\mathbf{C}^4_w/\tau)$. The improvement comes from the following estimate. For its proof,
see \cite[Lemma 6.1]{BC}.

\begin{lem}\label{pl2}
Let $\mathcal B$ be a $\mathbf C_w$-quasi-greedy basis. Fix $\tau\in (0,1]$ and
$x\in \mathbb{X}$. Let $A_1$ and $A_2$ be finite sets of positive integers such that $A_1\subset A_2$ and for all $n\in A_2$, we have $\tau\le |e_n^*(x)|\le 1$. Then
$$\|P_{A_1}x\|\ \le\ \frac{8\mathbf C_w^3}{\tau}\|P_{A_2}x\|.$$
\end{lem}

\begin{proof}[Proof of Theorem \ref{pst2}]
Let $x\in \mathbb{X}$, $m\in \mathbb{N}$, and $A\in \mathcal G(x, m,\tau)$. Set $\alpha:= \max_{n\notin A}|e_n^*(x)|$. The idea is to sandwich our $\tau$-weak greedy set $A$ between two greedy sets. If $\alpha = 0$, then $x - P_A(x) = 0$ and there is nothing to prove. Assume $\alpha > 0$. 
Define two sets
\begin{align*}
    A_1\ =\ \{n: |e^*_n(x)| > \alpha\}\mbox{ and } A_2\ =\ \{n: |e^*_n(x)|\ge \tau\alpha\}.
\end{align*}
Clearly, $A\subset A_2$, and since $A_1\cap A^c = \emptyset$, we know that $A_1 \subset A$. We write 
\begin{equation}\label{pe10}P_A(x)\ =\ P_{A_1}(x) + P_{A}(P_{A_2\backslash A_1}(x)).\end{equation}
Since $A_1$ is a greedy set of $x$, by $\mathbf{C}_w$-quasi-greediness, we obtain 
\begin{equation}\label{pe11}\|P_{A_1}(x)\|\ \le\ \mathbf C_w\|x\|.\end{equation}
For each $n\in A_2\backslash A_1$, we have $\tau\le \frac{|e^*_n(x)|}{\alpha}\le 1$. Due to Lemma \ref{pl2}, we obtain
\begin{align}\label{pe12}\|P_{A}(P_{A_2\backslash A_1}(x))\|&\ \le\ \frac{8\mathbf C_w^3}{\tau}\|P_{A_2\backslash A_1}(x)\|\ =\ \frac{8\mathbf C_w^3}{\tau}\|P_{A_2} (x) - P_{A_1}(x)\|\nonumber \\
&\ \le\ \frac{8\mathbf C_w^3}{\tau}\left(\|P_{A_2}(x)\| + \|P_{A_1}(x)\|\right)\ \le\ \frac{16\mathbf C_w^4}{\tau}\|x\|.
\end{align}
Together, \eqref{pe10}, \eqref{pe11}, and \eqref{pe12} give
$$\|P_A(x)\|\ \le\ \left(\mathbf C_w + \frac{16\mathbf C_w^4}{\tau}\right)\|x\|.$$
Therefore, $\|x-P_A (x)\| \le \left(1+\mathbf C_w + \frac{16\mathbf C_w^4}{\tau}\right)\|x\|$. 
\end{proof}

\subsection{From almost greedy to $(\mathbf C_{a, \tau},\tau)$-almost greedy}

The improvement in the upper bound in Theorem \ref{pst2} leads to an improvement for the theorem below, which shows that an almost greedy basis is $(\mathbf C_{a,\tau},\tau)$-almost greedy for any fixed $\tau\in (0,1]$. First, we state an useful estimate. 

\begin{lem}\label{gu}
Suppose that $\mathcal{B}$ is $\mathbf C_w$-quasi-greedy. Let $x\in \mathbb{X}$. For finite sets $A, B\subset \mathbb{N}$ with $A\subset B$, $(a_n)_{n\in A}\subset \mathbb{F}$, and any sign $(\varepsilon)$, we have that
\begin{enumerate}
    \item \begin{equation}\label{e15b}\left\|\sum_{n\in A}a_n e_n\right\|\ \le\ 2\mathbf C_w\max_{n\in A}|a_n|\left\|1_{\varepsilon A}\right\|.\end{equation}
    \item if $A$ is a greedy set of $x$, then
\begin{equation}\label{pe16}\min_{n\in A}| e_n^*(x)|\left\|\sum_{n\in A}\delta_ne_n\right\|\ \le\ 2\min\{\mathbf C_\ell , \mathbf C_w\}\left\|x\right\|,\end{equation}
where $\delta_n  = \sgn(e^*_n(x))$.
\end{enumerate}
\end{lem}

\begin{proof}
For \eqref{e15b}, see \cite[Corollary 10.2.11]{AK}; for \eqref{pe16}, see \cite[Lemma 2.3]{BBG}.
\end{proof}

\begin{thm}[Konyagin and Temlyakov \cite{KT}]\label{pst3}
Let $\mathcal B$ be a $\mathbf C_a$-almost greedy M-basis. Then for $\tau\in (0,1]$, we have
\begin{equation}\label{pe13}\gamma_{m,\tau}(x)\ \le \ C(\mathbf C_a, \tau)\tilde{\sigma}_m(x), \forall x\in \mathbb{X}, \forall m\in \mathbb{N}.\end{equation}
\end{thm}

Following the proof by Konyagin and Temlyakov, we see that $\mathbf C_{a,\tau} \lesssim \mathbf C_a^9/\tau^2$. As a byproduct of Theorem \ref{pst2} and Lemma \ref{gu}, we shall show that $\mathbf C_{a,\tau}\lesssim \mathbf C_a^7/\tau^2$. 

\begin{proof}
Let $x\in \mathbb{X}$, $m\in\mathbb{N}$, and $\varepsilon >0$. Let $A\subset\mathbb{N}$ be such that $|A| = m$ and $\|x-P_Ax\|\le \tilde{\sigma}_m(x) + \varepsilon$. Pick $B\in \mathcal G(x, m,\tau)$. By the triangle inequality, we have
\begin{align}\label{e091201}\|x-P_B(x)\| &\ \le\ \|x-P_A(x)\| + \|P_A(x) - P_B(x)\|\nonumber\\
&\ =\ \|x-P_A(x)\|+\|P_{A\backslash B}(x) - P_{B\backslash A}(x)\|\nonumber\\
&\ \le\ \|x-P_A(x)\|+\|P_{A\backslash B}(x)\| + \|P_{B\backslash A}(x)\|.
\end{align}
Note that $|A\backslash B| = |B\backslash A|$. If $|B\backslash A| = 0$, then $\|x-P_B(x)\|= \|x-P_A(x)\|\le \tilde{\sigma}_m(x) + \varepsilon$. Assume that $|B\backslash A|\neq 0$. We bound $\|P_{B\backslash A}(x)\|$. Set $y := x-P_A(x)$. Then $B\backslash A\in \mathcal{G}(y, |B\backslash A|,\tau)$.
Since a $\mathbf C_a$-almost greedy basis is $\mathbf C_a$-suppression quasi-greedy (by Theorem \ref{m7}), we can use Theorem \ref{pst2} to obtain
\begin{equation}\label{e091202}
    \|P_{B\backslash A}(x)\| \ =\ \|P_{B\backslash A}(y)\|\ \le\ C_1(\mathbf C_a, \tau)\|y\| \ =\ C_1(\mathbf C_a,\tau)\|x-P_A(x)\|,
\end{equation}
where $C_1(\mathbf C_a, \tau) \lesssim \mathbf C_a^4/\tau$.

Next, we bound $\|P_{A\backslash B}(x)\|$. By Theorem \ref{m7}, our basis has $\mathbf C_a$-property (A).
Observe that 
$$\tau\max_{n\in A\backslash B}|e^*_n(x)|\ \le\ \tau\max_{n\notin B}|e_n^*(x)|\ \le\ \min_{n\in B}|e^*_n(x)|\ \le\ \min_{n\in B\backslash A}|e^*_n(x)|.$$
Let $\varepsilon_n = \sgn(e_n^*(x))$ for all $n\in \supp(x)$. We have
\begin{align}\label{e091203}
\|P_{A\backslash B}(x)\| &\ \le\ 2(\mathbf C_a+1)\max_{n\in A\backslash B}|e_n^*(x)|\left\|\sum_{n\in A\backslash B}\varepsilon_n e_n\right\| \mbox{ due to }\eqref{e15b} \nonumber\\
&\ \le\ \frac{2(\mathbf C_a+1)\mathbf C_a}{\tau}\min_{n\in B\backslash A}|e_n^*(x)|\left\|\sum_{n\in B\backslash A}\varepsilon_n e_n\right\|\mbox{ due to $\mathbf C_a$-property (A)}\nonumber\\
&\ \le\ \frac{4(\mathbf C_a+1)\mathbf C^2_a}{\tau}\left\|P_{B\backslash A} (x)\right\| \mbox{ due to \eqref{pe16}}\nonumber\\ 
&\ \le\ \frac{4(\mathbf C_a+1)\mathbf C^2_a}{\tau}C_1(\mathbf C_a, \tau)\|x-P_A(x)\| \ \lesssim\  \frac{\mathbf C_a^7}{\tau^2}\|x-P_A(x)\|.
\end{align}
From \eqref{e091201}, \eqref{e091202}, and \eqref{e091203}, we obtain 
\begin{align*}
    \|x-P_B (x)\|\ \lesssim\ \frac{\mathbf C_a^7}{\tau^2}(\tilde{\sigma}_m(x) + \varepsilon).
\end{align*}
Letting $\varepsilon\rightarrow 0$, we are done. 
\end{proof}

\subsection{From greedy to $(\mathbf C_{g,\tau},\tau)$-greedy bases}
Temlyakov \cite{T} proved that when $\mathbb{X} = L_p, 1 < p <\infty$, the normalized Haar basis is $(\mathbf C_{g,\tau},\tau)$-greedy for all $\tau\in (0,1]$. Later, Konyagin and Temlyakov \cite{KT} noted that the same argument works for any greedy basis. In particular, we have
\begin{thm}[Konyagin and Temlyakov]\label{pst4}
Let $\mathcal B$ be a $\mathbf C_g$-greedy M-basis. Then for any fixed $\tau\in (0,1]$, we have
\begin{equation}\label{pe26}\gamma_{m,\tau}(x)\ \le \ \left( \frac{\mathbf C^4_g}{\tau}+\mathbf C_g\right)\sigma_m(x), \forall x\in \mathbb{X}, \forall m\in \mathbb{N}.\end{equation}
\end{thm}

\begin{proof}
Let $x\in \mathbb{X}$, $m\in\mathbb{N}$, and $\varepsilon >0$. Let $A\subset\mathbb{N}$ and $(a_n)_{n\in A}$ be such that $|A| = m$ and $\|x-\sum_{n\in A} a_n e_n\|\le \sigma_m(x) + \varepsilon$. Pick $B\in \mathcal G(x, m,\tau)$. By the triangle inequality, we have
\begin{align}\label{pe27}\|x-P_B(x)\| &\ =\ \|P^c_{A\cup B}(x) + P_{A\backslash B}(x)\|\ \le\ \|P^c_{A\cup B}(x)\| + \|P_{A\backslash B}(x)\|.
\end{align}

We bound $\|P_{A\backslash B}(x)\|$. By Theorem \ref{m6}, our basis is $\mathbf C_g$-suppression unconditional and satisfies $\mathbf C_g$-property (A).
Observe that 
$$\tau\max_{n\in A\backslash B}|e^*_n(x)|\ \le\ \tau\max_{n\notin B}|e_n^*(x)|\ \le\ \min_{n\in B}|e^*_n(x)|\ \le\ \min_{n\in B\backslash A}|e^*_n(x)|.$$
Let $\varepsilon_n = \sgn(e_n^*(x))$ for all $n\in \supp(x)$. We have
\begin{align}\label{pe29}
\|P_{A\backslash B}(x)\| &\ \le\ \mathbf C_g\max_{n\in A\backslash B}|e^*_n(x)|\left\|1_{\varepsilon A\backslash B}\right\|\mbox{ due  to Proposition \ref{p1}}\nonumber\\
&\ \le\ \frac{\mathbf C_g^2}{\tau}\min_{n\in B\backslash A} |e_n^*(x)|\left\|1_{\varepsilon B\backslash A}\right\|\mbox{ by $\mathbf C_g$-property (A)} \nonumber\\
&\ \le\ \frac{\mathbf C_g^3}{\tau}\left\|P_{B\backslash A}(x)\right\|\mbox{ due  to Proposition \ref{p1}}.
\end{align}

Next, we bound $\|P_{B\backslash A}(x)\|$ and $\|P^c_{A\cup B}(x)\|$ as follows:
\begin{align}
    &\|P_{B\backslash A}(x)\|\ =\ \left\|P_{B\backslash A}\left(x-\sum_{n\in A}a_ne_n\right)\right\|\ \le\ \mathbf C_g\left\|x-\sum_{n\in A}a_ne_n\right\|,\mbox{ and }\\
    &\|P^c_{A\cup B}(x)\|\ =\ \left\|P^c_{B\cup A}\left(x-\sum_{n\in A}a_ne_n\right)\right\|\ \le\ \mathbf C_g\left\|x-\sum_{n\in A}a_ne_n\right\|\label{e105}.
\end{align}
From \eqref{pe27} to \eqref{e105}, we obtain 
\begin{align*}
    \|x-P_B(x)\|\ \le\ \left( \frac{\mathbf C^4_g}{\tau}+\mathbf C_g\right)(\sigma_m(x) + \varepsilon).
\end{align*}
Letting $\varepsilon\rightarrow 0$, we are done. 
\end{proof}

\section{Future research}\label{future}

We list several problems for future research:

\begin{prob}\label{powerconst}\normalfont
Can we improve the power $1/2$ in \eqref{e26}?
\end{prob}

\begin{prob}\normalfont
Is there a Schauder basis $\mathcal{B}$ such that $\mathbf C_{b,\tau} < \infty$ for all $\tau\in (0,1]$ (see Definition \ref{d31}) and $\sup_\tau \mathbf C_{b,\tau} = \infty$? By Corollary \ref{ptu}, $\mathcal{B}$ must be non-quasi-greedy. 
\end{prob}

\begin{prob}\normalfont
Given $0 < \tau_2 < \tau_1 \le 1$, for an M-basis, does property (A, $\tau_1$) imply property (A, $\tau_2$) without any additional assumption? 
\end{prob}

\begin{prob}\normalfont
Is the estimate in Proposition \ref{pp1} optimal? Specifically, for $0 < \tau_2 <\tau_1\le 1$, can we construct a basis $\mathcal{B}$ with $\mathbf {C}_{b,\tau_2}$-property (A, $\tau_2$) and $\mathbf C_{b,\tau_1}$-property (A, $\tau_1$) so that $\mathbf C_{b,\tau_1} = \mathbf C_{b,\tau_2}\tau_1/\tau_2$?
\end{prob}

\appendix
\section{Repeated proofs}
\begin{proof}[Proof of Lemma \ref{l2}]
Let $A, B, x, (\varepsilon), t$ be chosen as in \eqref{e8}.  By norm convexity, we have
\begin{align*}
    \|x\|\ =\ \left\|x-P_B(x) + \sum_{n\in B}e_n^*(x)e_n\right\| &\ \le\ \sup_{(\delta)}\left\|x-P_B(x) + \frac{t}{\tau} 1_{\delta B}\right\|\\
    &\ =\ \frac{1}{\tau}\sup_{(\delta)}\left\|\tau(x-P_B(x)) + t 1_{\delta B}\right\|\\
    &\ \le\ \frac{\nu_{m,\tau,\ell}}{\tau}\left\|x-P_B(x) + t 1_{\varepsilon A}\right\|.
\end{align*}
Hence, $\Omega_{m,\tau,\ell}\le \nu_{m,\tau,\ell}/\tau$.

For the reverse inequality, let $A, B, x, (\varepsilon), (\delta)$ be chosen as in \eqref{e7}. Let $y = \tau x + 1_{\delta B}$. Then $\|y\|_\infty \le 1$. By \eqref{e5}, 
$$\|\tau x + 1_{\delta B}\|\ =\ \|y\| \ \le\ \Omega_{m,\tau,\ell}\|y-P_B(y) + \tau 1_{\varepsilon A}\|\ =\ \Omega_{m,\tau,\ell}\tau \|x + 1_{\varepsilon A}\|.$$
Hence, $\Omega_{m,\tau,\ell}\ge \nu_{m,\tau,\ell}/\tau$.
\end{proof}

\begin{proof}[Proof of Lemma \ref{l6}]
Let $A, B, x, (\varepsilon), t$ be chosen as in \eqref{e16}. By norm convexity, we have
\begin{align*}\|x\|\ =\ \left\|x-P_B(x) + \sum_{n\in B}e_n^*(x)e_n\right\|&\ \le\ \sup_{(\delta)}\left\|x-P_B(x) + \frac{t}{\tau}1_{\delta B}\right\|\\
&\ =\ \frac{1}{\tau}\sup_{(\delta)}\|\tau(x-P_B(x))+t1_{\delta B}\|\\
&\ \le\ \frac{\nu'_{m, \tau, \ell}}{\tau}\sup_{(\delta)}\|x-P_B(x)+t1_{\varepsilon A}\|.
\end{align*}
Hence, $\Omega'_{m,\tau, \ell}\le \nu'_{m,\tau,\ell}/\tau$.

For the reverse inequality, let $A, B, x, (\varepsilon), (\delta)$ be chosen as in \eqref{e13}. Let $y = \tau x + 1_{\delta B}$. Then $\|y\|_\infty \le 1$. We have
\begin{align*}
    \|\tau x + 1_{\delta B}\| \ =\ \|y\| \ \le\ \Omega'_{m,\tau,\ell}\|y-P_B(y) + \tau 1_{\varepsilon A}\|\ =\ \tau\Omega'_{m,\tau,\ell}\|x+ 1_{\varepsilon A}\|.
\end{align*}
Hence, $\Omega'_{m,\tau, \ell}\ge \nu'_{m,\tau, \ell}/\tau$.
\end{proof}

\ \\
\end{document}